\RequirePackage{fix-cm}
\documentclass[smallextended]{svjour3}       
\smartqed  
\usepackage{graphicx}
\usepackage{latexsym}
\usepackage{amsmath}
\usepackage{amsfonts}
\usepackage{amssymb}
\usepackage{algorithm}
\usepackage[noend]{algpseudocode}
\usepackage{cite}
\usepackage{enumitem}
\usepackage{float}
\usepackage{mathtools}
\usepackage{tabularx}
\RequirePackage[colorlinks,citecolor=blue,urlcolor=blue]{hyperref}
\newcommand{\abs}[1]{\ensuremath{\left|#1\right|}}

\newcommand{\norm}[1]{\ensuremath{\left\|#1\right\|}}

\newtheorem{assumption}{Assumption}

\DeclareMathOperator{\Tr}{Tr}

\journalname{}
\begin{document}

\title{Secant Penalized BFGS: A Noise Robust Quasi-Newton Method Via Penalizing The Secant Condition
}

\titlerunning{Secant Penalized BFGS}        

\author{Brian Irwin*\thanks{* Corresponding author.} \and Eldad Haber}

\authorrunning{Brian Irwin \and Eldad Haber} 

\institute{{Brian Irwin \href{https://orcid.org/0000-0002-6086-4359}{\includegraphics[scale=0.4]{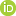}} \and Eldad Haber}
	      \at
              Department of Earth, Ocean and Atmospheric Sciences, The University of British Columbia, Vancouver, BC, Canada \\
              \email{\{birwin, haber\}@eoas.ubc.ca}
}

\date{Received: date / Accepted: date}

\maketitle

\begin{abstract}
In this paper, we introduce a new variant of the BFGS method designed to perform well when gradient measurements are corrupted by noise. We show that by treating the secant condition with a penalty method approach motivated by regularized least squares estimation, one can smoothly interpolate between updating the inverse Hessian approximation with the original BFGS update formula and not updating the inverse Hessian approximation. Furthermore, we find the curvature condition is smoothly relaxed as the interpolation moves towards not updating the inverse Hessian approximation, disappearing entirely when the inverse Hessian approximation is not updated. These developments allow us to develop a method we refer to as secant penalized BFGS (SP-BFGS) that allows one to relax the secant condition based on the amount of noise in the gradient measurements. SP-BFGS provides a means of incrementally updating the new inverse Hessian approximation with a controlled amount of bias towards the previous inverse Hessian approximation, which allows one to replace the overwriting nature of the original BFGS update with an averaging nature that resists the destructive effects of noise and can cope with negative curvature measurements. We discuss the theoretical properties of SP-BFGS, including convergence when minimizing strongly convex functions in the presence of uniformly bounded noise. Finally, we present extensive numerical experiments using over 30 problems from the CUTEst test problem set that demonstrate the superior performance of SP-BFGS compared to BFGS in the presence of both noisy function and gradient evaluations.
\keywords{Quasi-Newton Methods \and Secant Condition \and Penalty Methods \and Least Squares Estimation \and Measurement Error \and Noise Robust Optimization}
\end{abstract}

\section{Introduction}
\label{sec:intro}
Over the past 50 years, quasi-Newton methods have proved to be some of the most economical and effective methods for a variety of optimization problems. Originally conceived to provide some of the advantages of second order methods without the full cost of Newton's method, quasi-Newton methods, which are also referred to as variable metric methods \cite{Johnson2019_quasiNewton_notes}, are based on the observation that by differencing observed gradients, one can calculate approximate curvature information.  This approximate curvature information can then be used to improve the speed of convergence, especially in comparison to first order methods, such as gradient descent. There are currently a variety of different quasi-Newton methods, with the Broyden-Fletcher-Goldfarb-Shanno (BFGS) method \cite{10.1093/imamat/6.1.76, 10.1093/comjnl/13.3.317, 10.2307/2004873, 10.2307/2004840} almost certainly being the best known quasi-Newton method. 

Modern quasi-Newton methods were developed for problems involving the optimization of smooth functions without constraints. The BFGS method is the best known quasi-Newton method because in practice it has demonstrated superior performance due to its very effective self-correcting properties \cite{Nocedal2006}. Accordingly, BFGS has since been extended to handle box constraints \cite{doi:10.1137/0916069}, and shown to be effective even for some nonsmooth optimization problems \cite{Lewis2013NonsmoothOV}. Furthermore, a limited memory version of BFGS known as L-BFGS \cite{Liu1989} has become a favourite algorithm for solving optimization problems with a very large number of variables, as it avoids directly storing approximate inverse Hessian matrices. However, BFGS and its relatives were not designed to explicitly handle noisy optimization problems, and noise can unacceptably degrade the performance of these methods. 

The authors of \cite{doi:10.1137/140954362} make the important observation that quasi-Newton updating is inherently an overwriting process rather than an averaging process. Fundamentally, differencing noisy gradients can produce harmful efffects because the resulting approximate curvature information may be inaccurate, and this inaccurate curvature information may overwrite accurate curvature information. Newton's method can naturally be viewed as a local rescaling of coordinates so that the rescaled problem is better conditioned than the original problem. Quasi-Newton methods attempt to perform a similar rescaling, but instead of using the (inverse) Hessian matrix to obtain curvature information for the rescaling, they use differences of gradients to obtain curvature information. Thus, it should be unsurprising that inaccurate curvature information obtained from differencing noisy gradients can be problematic because it means the resulting rescaling of the problem can be poor, and the conditioning of the rescaled problem could be even worse than the conditioning of the original problem.

With the above in mind, several works have dealt with how to improve the performance of quasi-Newton methods in the presence of noise. Many recent works focus on the empirical risk minimization (ERM) problem, which is ubiquitous in machine learning. For example, in \cite{doi:10.1137/140954362} the authors propose a technique designed for the stochastic approximation (SA) regime that employs subsampled Hessian-vector products to collect curvature information pointwise and at spaced intervals, in contrast to the classical approach of computing the difference of gradients at each iteration. This work is built upon in \cite{pmlr-v51-moritz16}, where the authors present a stochastic L-BFGS algorithm that draws upon the variance reduction approach of \cite{Johnson2013AcceleratingSG}. In \cite{doi:10.1137/15M1053141}, the authors outline a stochastic damped limited-memory BFGS (SdLBFGS) method that employs damping techniques used in sequential quadratic programming (SQP). A stochastic block BFGS method that updates the approximate inverse Hessian matrix using a sketch of the Hessian matrix is proposed in \cite{pmlr-v48-gower16}. Further work on stochastic L-BFGS algorithms, including convergence results, can be found in \cite{pmlr-v2-schraudolph07a, 10.5555/2789272.2912100, Zhao2018StochasticLI, 8626766}. 

Despite the importance of the ERM problem due to the current prevalence of machine learning, there are still a variety of important noisy optimization problems that arise in other contexts. In engineering design, numerical simulations are often employed in place of conducting costly, if even feasible, physical experiments. In this context, one tries to find optimal design parameters using the numerical simulation instead of physical experiments. Some examples from aerospace engineering, including interplanetary trajectory and wing design, can be found in \cite{Fasano2019, doi:10.1002/0470855487, doi:10.2514/1.J057294}. Examples from materials engineering include stable composite design \cite{doi:10.1177/1464420716664921} and ternary alloy composition \cite{Graf2017}, amongst others \cite{Munoz-Rojas2016}, while examples from electrical engineering include power system operation \cite{doi:10.1002/9780470466971}, hardware verification \cite{Gal_2020_HowToCatchALion}, and antenna design \cite{Koziel2014}. Noise is often an unavoidable property of such numerical simulations, as the simulations can include stochastic internal components, and floating point arithmetic vulnerable to roundoff error. Apart from the analysis of the BFGS method with bounded errors in \cite{doi:10.1137/19M1240794}, there is relatively little work on the behaviour of quasi-Newton methods in the presence of general bounded noise. As optimizing noisy numerical simulations does not always fit the framework of the ERM problem, analyses of the behaviour of quasi-Newton methods in the presence of general bounded noise are of practical value when optimizing numerical simulations. 

\subsection{Contributions}

Noise is inevitably introduced into machine learning problems due to the approximations required to handle large datasets, and numerical simulations due to the effects of finite precision arithmetic, and parts of the simulator containing inherently stochastic components. In this paper, we return to the fundamental theory underlying the design of quasi-Newton methods, which allows us to design a new variant of the BFGS method that explicitly handles the corrupting effects of noise. We do this as follows:
\begin{enumerate} 

	\item
	In Section~\ref{sec:mathematical-background}, we review the setup and derivation of the original BFGS method.
	
	\item
	In Section~\ref{sec:SPBFGS-derivation}, motivated by regularized least squares estimation, we treat the secant condition of BFGS with a penalty method. This creates a new BFGS update formula that we refer to as secant penalized BFGS (SP-BFGS), which we show reduces to the original BFGS update formula in a limiting case, as expected.

	\item
	In Section~\ref{sec:algorithmic-framework}, we present an algorithmic framework for practically implementing SP-BFGS updating. We also discuss implementation details, including how to perform a line search and choose the penalty parameter in the presence of noise. 

	\item
	In Section~\ref{sec:convergence-analysis}, we discuss the theoretical properties of SP-BFGS, including how the penalty parameter influences the eigenvalues of the approximate inverse Hessian. This allows us to show that under appropriate conditions SP-BFGS iterations are guaranteed to converge linearly to a neighborhood of the global minimizer when minimizing strongly convex functions in the presence of uniformly bounded noise. 
	
	\item
	In Section~\ref{sec:numerical-experiments}, we study the empirical performance of SP-BFGS updating compared to BFGS updating by performing extensive numerical experiments with both convex and nonconvex objective functions corrupted by function and gradient noise. Results from a diverse set of over 30 problems from the CUTEst test problem set demonstrate that intelligently implemented SP-BFGS updating frequently outperforms BFGS updating in the presence of noise.
	
	\item 
	Finally, Section~\ref{sec:final-remarks} concludes the paper and outlines directions for further work. 
		
\end{enumerate}

\section{Mathematical Background}
\label{sec:mathematical-background}
In this section, as preliminaries to the main results of this paper, we review the setup and derivation of the original BFGS method.

\subsection{BFGS Setup}
\label{sec:BFGS-Setup}

The BFGS method was originally designed to solve the following unconstrained optimization problem 
\begin{equation} \label{eq:phi}
\min_{x} \big \{ \phi(x) \big \}
\end{equation}
with $x \in \mathbb{R}^{n}$, $\phi: \mathbb{R}^{n} \mapsto \mathbb{R}$, and $\phi$ being a smooth twice continuously differentiable and nonnoisy function. Below, we use the notational conventions of \cite{Nocedal2006}, including $\phi_k = \phi(x_k)$. We begin by using the Taylor expansion of $\phi$ to build a local quadratic model $m_{k}$ of the objective function $\phi$ at the $k^{th}$ iterate $x_{k}$ of the optimization procedure
\begin{equation}
\phi(x_k + p) \approx \phi_{k} + \nabla \phi_{k}^{T} p + \frac{1}{2} p^{T} B_{k} p = m_{k}(p)
\end{equation} 
where $B_{k}$ is an $n \times n$ symmetric positive definite matrix that approximates the Hessian matrix (i.e. $B_{k} \approx \nabla^{2} \phi_{k}$). By setting the gradient of $m_{k}$ to zero, we see that the unique minimizer $p_{k}$ of this local quadratic model is 
\begin{equation}
p_{k} = - B_{k}^{-1} \nabla \phi_{k}
\end{equation}
and thus it is natural to update the next iterate $x_{k+1}$ as
\begin{equation} \label{eq:line-search-iteration}
x_{k+1} = x_{k} + \alpha_{k} p_{k}
\end{equation}
where $\alpha_{k}$ is the step size along the direction $p_{k}$, which is often chosen using a line search. 

To avoid computing $B_{k}$ from scratch at each iteration $k$, we use the curvature information from recent gradient evaluations to update $B_{k}$, and thus relatively economically form $B_{k+1}$. A Taylor expansion of $\nabla \phi$ reveals 
\begin{equation}
\nabla \phi(x_k + p) \approx \nabla \phi_k + \nabla^2 \phi_k p 
\end{equation}
and so it is reasonable to require that the new approximate Hessian $B_{k+1}$ satisfies
\begin{equation}
\nabla \phi_{k+1} = \nabla \phi_{k} + \alpha_{k} B_{k+1} p_{k}
\end{equation}
which rearranges to
\begin{equation}
B_{k+1} \alpha_{k} p_{k} = \nabla \phi_{k+1} - \nabla \phi_{k}.
\end{equation}
Now, define the two new quantities $s_{k}$ and $y_{k}$ as 
\begin{subequations}
\begin{equation} \label{eq:s_k-definition}
s_{k} \coloneqq x_{k+1} - x_{k} = \alpha_{k} p_{k}, 
\end{equation} 
\begin{equation} \label{eq:y_k-definition}
y_{k} \coloneqq \nabla \phi_{k+1} - \nabla \phi_{k}.
\end{equation}
\end{subequations}
Thus, we arrive at (\ref{eq:secant_B}), which is known as the \textit{secant condition}
\begin{equation} \label{eq:secant_B}
B_{k+1} s_{k} = y_{k}.
\end{equation}
In words, the secant condition dictates that the new approximate Hessian $B_{k+1}$ must map the measured displacement $s_{k}$ into the measured difference of gradients $y_{k}$. If we denote the approximate inverse Hessian $H_{k} = B_{k}^{-1} \approx \nabla^{2} \phi_{k}^{-1}$, then the secant condition can be equivalently expressed as (\ref{eq:secant_H})
\begin{equation} \label{eq:secant_H}
H_{k+1} y_{k} = s_{k}.
\end{equation}

As $H_{k+1}$ is not yet uniquely determined, to obtain the BFGS update formula, we impose a minimum norm restriction. Specifically, we choose $H_{k+1}$ to be the solution of the following quadratic program over matrices
\begin{equation} \label{eq:quad_program}
\min_{H} \bigg \{ \frac{1}{2} \norm{W^{1/2} (H - H_{k}) W^{1/2}}_{F}^{2} \bigg \} \quad \text{s.t.} \quad H = H^{T}, \text{~~~} H y_{k} = s_{k}
\end{equation} 
where $|| \cdot ||_{F}$ denotes the Frobenius norm, and $W^{1/2}$ the principal square root (see \cite{MatrixAnalysisHornJohnson} or a similar reference) of a symmetric positive definite weight matrix $W$ satisfying 
\begin{equation} \label{eq:W}
W s_{k} = y_{k}.
\end{equation}
As we will see, choosing the weight matrix $W$ to satisfy (\ref{eq:W}) ensures that the resulting optimization method is scale invariant. The weight matrix $W$ can be chosen to be any symmetric positive definite matrix satisfying (\ref{eq:W}), and the specific choice of $W$ is not of great importance, as $W$ will not appear directly in the main results of this paper. However, as a concrete example from \cite{Nocedal2006}, one could assume $W = \bar{G}_{k}$, where $\bar{G}_{k}$ is the average Hessian defined by
\begin{equation} \label{eq:average-Hessian-G}
\bar{G}_{k} = \int_{0}^{1} \nabla^2 \phi(x_k + t \alpha_k p_k) d t  \text{~~}.
\end{equation}

\subsection{Solving For The BFGS Update}
\label{sec:Solving-For-BFGS-Update}
To solve the quadratic program given by (\ref{eq:quad_program}), we setup a Lagrangian $\mathcal{L}(H,q,\Gamma)$ involving the constraints. Recalling that
\begin{equation}
\norm{W^{1/2} (H - H_{k}) W^{1/2}}^{2}_{F} = \Tr \bigg ( W (H - H_{k}) W (H - H_{k})^T \bigg )  \text{~},
\end{equation}
this gives the Lagrangian defined by (\ref{eq:BFGS-lagrangian}) below
\small
\begin{equation} \label{eq:BFGS-lagrangian}
\mathcal{L} = \frac{1}{2} \Tr \bigg (W (H - H_{k}) W (H - H_{k})^T \bigg ) + \Tr \bigg ( (H y_k - s_k) q^T \bigg ) + \Tr \bigg ( \Gamma (H - H^T ) \bigg )  
\end{equation}
\normalsize
where $q$ is a vector of Lagrange multipliers associated with the secant condition, and $\Gamma$ is a matrix of Lagrange multipliers associated with the symmetry condition. Taking the derivative of the Lagrangian $\mathcal{L}(H,q,\Gamma)$ with respect to the matrix $H$ yields
\begin{equation}
\frac{\partial \mathcal{L}(H,q,\Gamma)}{\partial H} = W (H - H_k) W + q y_k^T + \Gamma^T - \Gamma 
\end{equation} 
and so we have the Karush-Kuhn-Tucker (KKT) system defined by the three equations~(\ref{eq:BFGS_lagrange_FOC_KKT}),~(\ref{eq:BFGS_sec_KKT}), and~(\ref{eq:BFGS_sym_KKT}) below
\begin{subequations}
\begin{equation} \label{eq:BFGS_lagrange_FOC_KKT}
W (H - H_k) W + q y_k^T + \Gamma^T - \Gamma = 0  \text{~},
\end{equation} 
\begin{equation} \label{eq:BFGS_sec_KKT}
H y_{k} - s_{k} = 0  \text{~},
\end{equation}
\begin{equation} \label{eq:BFGS_sym_KKT}
H - H^{T} = 0  \text{~}. 
\end{equation} 
\end{subequations}
For brevity, we omit the details of the solution of the KKT system defined above because it is a limiting case of the system solved in Theorem~\ref{thm:sp-bfgs-update}. For an alternative geometric solution technique, we refer the interested reader to Section 2 of \cite{doi:10.1080/10556780802367205}. The minimizer $H^{*} = H_{k+1}$ is given by the well known BFGS update formula
\begin{equation} \label{eq:BFGS-Direct-Update}
H_{k+1} = \bigg ( I - \frac{s_k y_k^T}{s_k^T y_k} \bigg ) H_k \bigg ( I - \frac{y_k s_k^T}{s_k^T y_k} \bigg ) + \frac{s_k s_k^T}{s_k^T y_k}
\end{equation}
which, if we define the curvature parameter $\rho_{k} = \frac{1}{s_{k}^T y_{k}}$, can be equivalently written as
\begin{equation} \label{eq:BFGS-Direct-Update-Rho}
H_{k+1} = \bigg ( I - \rho_k s_k y_k^T \bigg ) H_k \bigg ( I - \rho_k y_k s_k^T \bigg ) + \rho_k s_k s_k^T.
\end{equation} 

Applying the Sherman-Morrison-Woodbury formula (see \cite{10.2307/2030425}) to the BFGS update formula immediately above, one can also write the BFGS update in terms of the approximate Hessian $B_{k} = H_{k}^{-1}$ instead of the approximate inverse Hessian. Again, for brevity, the details are omitted because they are a special case of Theorem~\ref{thm:sp-bfgs-inverse-update} shown later. The result is
\begin{equation} \label{eq:BFGS-Inverse-Update}
B_{k+1} = B_k - \frac{B_k s_k s_k^T B_k}{s_k^T B_k s_k} + \frac{y_k y_k^T}{s_k^T y_k} = B_k - \frac{B_k s_k s_k^T B_k}{s_k^T B_k s_k} + \rho_k y_k y_k^T.
\end{equation}

To ensure the updated approximate Hessian $B_{k+1}$ is positive definite, we must enforce that 
\begin{equation} \label{eq:B-pd}
s_k^T B_{k+1} s_k > 0 .
\end{equation}
Substituting $B_{k+1} s_k = y_k$ from the secant condition, the condition~(\ref{eq:B-pd}) becomes
\begin{equation} \label{eq:curvature-condition}
s_k^T y_k > 0
\end{equation}
which is known as the \textit{curvature condition}, as it is equivalent to 
\begin{equation} \label{eq:curvature-condition-rho}
\frac{1}{\rho_k} > 0 .
\end{equation}

\section{Derivation Of Secant Penalized BFGS}
\label{sec:SPBFGS-derivation}
In this section, having reviewed the construction of the original BFGS method, we now show how treating the secant condition with a penalty method approach motivated by regularized least squares estimation allows one to generalize the original BFGS update.

\subsection{Penalizing The Secant Condition}
\label{sec:Penalizing-The-Secant-Condition}
By applying a penalty method (see Chapter 17 of \cite{Nocedal2006}) to the secant condition instead of directly enforcing the secant condition as a constraint, we obtain the problem
\small
\begin{equation} \label{eq:SPBFGS-objective}
\min_{H} \bigg \{ \frac{1}{2} \norm{W^{1/2} (H - H_{k}) W^{1/2}}_{F}^{2} + \frac{\beta_k}{2} \norm{W^{1/2} (H y_k - s_k)}_2^2 \bigg \} \quad \text{s.t.} \quad H = H^{T} 
\end{equation} 
\normalsize
where $\beta_k \in [0, +\infty]$ is a penalty parameter that determines how strongly to penalize violations of the secant condition. As we will see, one recovers the solution to the constrained problem (\ref{eq:quad_program}) in the limit $\beta_k = +\infty$, so $\beta_k$ can be intuitively thought of as the cost of violating the secant condition. By treating the symmetry constraint with a matrix $\Gamma$ of Lagrange multipliers again, we obtain the following Lagrangian
\small
\begin{equation} \label{eq:SPBFGS-lagrangian}
\mathcal{L}= \frac{1}{2} \Tr \bigg (W (H - H_{k}) W (H - H_{k})^T \bigg ) + \frac{\beta_k}{2} \norm{W^{1/2} (H y_k - s_k)}_2^2 + \Tr \bigg ( \Gamma (H - H^T ) \bigg )  \text{~}.
\end{equation}
\normalsize
Defining the residual associated with the secant condition as $r_k(H) \coloneqq H y_k - s_k$ and $u \coloneqq \beta_k W r_k$, the first order optimality conditions of (\ref{eq:SPBFGS-lagrangian}) can be written as the system
\begin{subequations}
\begin{equation} \label{eq:SP-BFGS-lagrange-FOC}
W (H - H_k) W + u y_k^T + \Gamma^T - \Gamma = 0  \text{~},
\end{equation} 
\begin{equation} \label{eq:SP-BFGS-sec-cond}
H y_{k} - s_{k} - \frac{W^{-1} u}{\beta_k} = 0  \text{~},
\end{equation}
\begin{equation} \label{eq:SP-BFGS-sym-cond}
H - H^{T} = 0  \text{~}. 
\end{equation} 
\end{subequations}
Note that, as expected, in the limit $\beta_k = +\infty$, the system given by (\ref{eq:SP-BFGS-lagrange-FOC}), (\ref{eq:SP-BFGS-sec-cond}), and (\ref{eq:SP-BFGS-sym-cond}) reduces to the KKT system given by (\ref{eq:BFGS_lagrange_FOC_KKT}),~(\ref{eq:BFGS_sec_KKT}), and~(\ref{eq:BFGS_sym_KKT}).

We now find an explicit closed form solution to the problem given by (\ref{eq:SPBFGS-objective}), which is given in Theorem~\ref{thm:sp-bfgs-update}. 
\begin{theorem}[SP-BFGS Update] \label{thm:sp-bfgs-update}
The update formula given by the minimizer $H^{*}$ of the problem defined by (\ref{eq:SPBFGS-objective}), which can be obtained by solving the system given by (\ref{eq:SP-BFGS-lagrange-FOC}),~(\ref{eq:SP-BFGS-sec-cond}), and~(\ref{eq:SP-BFGS-sym-cond}), is the SP-BFGS update
\small
\begin{equation} \label{eq:SP-BFGS-Direct-Update-factored}
H_{k+1} = \bigg ( I - \omega_k s_k y_k^T \bigg ) H_k \bigg ( I - \omega_k y_k s_k^T \bigg ) + \omega_k \bigg [ \frac{\gamma_k}{\omega_k} + (\gamma_k - \omega_k) y_k^T H_k y_k \bigg ] s_k s_k^T 
\end{equation}
\normalsize
where 
\begin{equation} \label{eq:gamma-omega-definitions}
\gamma_k = \frac{1}{(s_k^T y_k + \frac{1}{\beta_k})}, \quad \omega_k = \frac{1}{(s_k^T y_k + \frac{2}{\beta_k})} .
\end{equation}
\end{theorem}
\begin{proof}
See Appendix~\ref{app:SPBFGS-update-proof}.
\smartqed
\end{proof}

At this point, a few comments are in order regarding the SP-BFGS update given by (\ref{eq:SP-BFGS-Direct-Update-factored}). First, observe that as $\beta_k \rightarrow +\infty$, we have that $\omega_k \rightarrow \rho_k$ and $\gamma_k \rightarrow \rho_k$. As a result, when $\beta_k = +\infty$, one recovers the original BFGS update, as expected. Second, also observe that as $\beta_k \rightarrow 0$, we have that $\omega_k \rightarrow 0$ and $\gamma_k \rightarrow 0$. As a result, we see that in the case $\beta_k = 0$ the SP-BFGS update reduces to $H_{k+1} = H_{k}$. This is again expected because as $\beta_k \rightarrow 0$, the cost of violating the secant condition goes to zero, and the minimum norm symmetric update is simply $H_{k+1} = H_{k}$.  

We now examine what the analog of the curvature condition (\ref{eq:curvature-condition}) is for SP-BFGS. Lemma~\ref{thm:sp-bfgs-curv-cond} demonstrates that (\ref{eq:sp-bfgs-curv-cond}) is the SP-BFGS analog of the BFGS curvature condition (\ref{eq:curvature-condition}).
\begin{lemma}[Positive Definiteness Of SP-BFGS Update] \label{thm:sp-bfgs-curv-cond}
If $H_{k}$ is positive definite, then the $H_{k+1}$ given by the SP-BFGS update (\ref{eq:SP-BFGS-Direct-Update-factored}) is positive definite if and only if the SP-BFGS curvature condition
\begin{equation} \label{eq:sp-bfgs-curv-cond}
s_k^T y_k > - \frac{1}{\beta_k} 
\end{equation}
is satisfied.
\end{lemma}
\begin{proof}
See Appendix~\ref{app:sp-bfgs-curv-cond-proof}.
\smartqed
\end{proof}

The result in Lemma~\ref{thm:sp-bfgs-curv-cond} warrants some discussion. First, the limiting behaviour with respect to $\beta_k$ is consistent with Theorem~\ref{thm:sp-bfgs-update}. As $\beta_k \rightarrow +\infty$, condition~(\ref{eq:sp-bfgs-curv-cond}) reduces to the BFGS curvature condition (\ref{eq:curvature-condition}). As $\beta_k \rightarrow 0$, condition~(\ref{eq:sp-bfgs-curv-cond}) reduces to no condition at all, as $s_k^T y_k > - \infty$ is always true. This is consistent with the observation that when $\beta_k = 0$, the minimum norm symmetric update is $H_{k+1} = H_{k}$, and in this case $H_{k+1}$ is guaranteed to be positive definite if $H_k$ is positive definite, regardless of $s_k^T y_k$. 

From the proof of Lemma~\ref{thm:sp-bfgs-curv-cond} (see (\ref{eq:convex-combination})), it is now clear that
\begin{equation}
y_k^T H_{k+1} y_{k} = \bigg ( \frac{\beta_k y_k^T s_{k}}{1 + \beta_k y_k^T s_{k}} \bigg ) y_k^T s_{k} + \bigg ( \frac{1}{1 + \beta_k y_k^T s_{k}} \bigg ) y_k^T H_k y_k
\end{equation}
and so $y_{k}^T H_{k+1} y_{k}$ is a convex combination of $y_{k}^T s_k$ and $y_{k}^T H_k y_k$. Thus, $H_{k+1}$ interpolates between the current inverse Hessian approximation $H_k$ and the original BFGS update, and as $\beta_k$ decreases, the interpolation is increasingly biased towards the current approximation $H_k$. From a regularized least squares estimation perspective, $\beta_k$ plays the role of a regularization parameter that controls the amount of bias in the estimate of $H_{k+1}$. Note that this behaviour is somewhat similar to the behaviour of Powell damping \cite{Powell1978}, although Powell damping was introduced to handle approximating a potentially indefinite Hessian of the Lagrangian in constrained optimization problems, and not noise. 

We finish introducing the SP-BFGS update by applying the Sherman-Morrison-Woodbury formula to~(\ref{eq:SP-BFGS-Direct-Update-factored}), which allows us to write the update in terms of the approximate Hessian $B_k$ instead of the approximate inverse Hessian $H_k$. The result is given in Theorem~\ref{thm:sp-bfgs-inverse-update}.
\begin{theorem}[SP-BFGS Inverse Update] \label{thm:sp-bfgs-inverse-update}
The SP-BFGS update formula given by~(\ref{eq:SP-BFGS-Direct-Update-factored}) can be written in terms of $B_{k} = H_{k}^{-1}$ as
\begin{equation*}
\resizebox{.95\hsize}{!}{$B_{k+1} = B_k - \frac{\omega_k \bigg [ \bigg ( (\omega_k - \gamma_k) y_k^T B_k^{-1} y_k - \frac{\gamma_k}{\omega_k} \bigg ) B_k s_k s_k^T B_k + (1 - \omega_k s_k^T y_k) ( B_k s_k y_k^T + y_k s_k^T B_k ) + \omega_k (s_k^T B_k s_k) y_k y_k^T \bigg ]}{\big ( (\omega_k - \gamma_k) y_k^T B_k^{-1} y_k - \frac{\gamma_k}{\omega_k} \big ) \big ( \omega_k s_k^T B_k s_k \big ) - (1 - \omega_k y_k^T s_k)^2} . $}
\end{equation*}
\end{theorem}
\begin{proof}
See Appendix~\ref{app:sp-bfgs-inverse-update-proof}.
\smartqed
\end{proof}

Note that the limiting behaviour of Theorem~\ref{thm:sp-bfgs-inverse-update} with respect to $\beta_k$ is again consistent. When $\beta_k = +\infty$, we obtain the original BFGS inverse update (\ref{eq:BFGS-Inverse-Update}), and when $\beta_k = 0$, we obtain $B_{k+1} = B_{k}$. One complication with respect to the SP-BFGS inverse update (\ref{eq:SP-BFGS-Inverse-Update}) is that $B_{k+1}$ cannot in general be expressed solely in terms of $B_k$ due to the presence of $y_k^T B_k^{-1} y_k$ (i.e. $y_k^T H_k y_k$) in the denominator.

\section{Algorithmic Framework}
\label{sec:algorithmic-framework}
We now outline how to practically implement SP-BFGS updating. We consider the situation where one has access to noise corrupted versions of a smooth function $\phi$ and its gradient $\nabla \phi$ that can be decomposed as 
\begin{equation} \label{eq:noisy-function-decomposition}
f(x) = \phi(x) + \epsilon(x) ,
\end{equation}
\begin{equation} \label{eq:noisy-gradient-decomposition}
g(x) = \nabla \phi(x) + e(x) .
\end{equation}
In~(\ref{eq:noisy-function-decomposition}) and~(\ref{eq:noisy-gradient-decomposition}), $\phi$ is a smooth twice continuously differentiable function as in Section~\ref{sec:BFGS-Setup}, and $\epsilon(x)$ is a scalar representing noise in the function evaluations. Similarly, $\nabla \phi$ is the gradient of the smooth function $\phi$, while $e(x)$ is a vector representing noise in the gradient evaluations. Similar decompositions are used in \cite{DFONoisyFunctionsQuasiNewton, Gal_2020_HowToCatchALion, doi:10.1137/19M1240794}. 

\subsection{Minimization Routine}
\label{sec:minimization-routine}

Algorithm~\ref{alg:SP-BFGS-routine} outlines a general procedure for minimizing a noisy function with noisy function and gradient values $f$ and $g$ that can be decomposed as shown in~(\ref{eq:noisy-function-decomposition}) and~(\ref{eq:noisy-gradient-decomposition}). The inputs to the procedure in Algorithm~\ref{alg:SP-BFGS-routine} are a means of evaluating the noisy objective function $f(x)$ and gradient $g(x)$, the starting point $x^0$, and an initial inverse Hessian approximation $H^0$. As the best convergence/stopping test is problem dependent, we note that standard gradient and function value based tests can be employed in conjuction with smoothing and noise estimation techniques (e.g. see Section 3.3.4 of \cite{DFONoisyFunctionsQuasiNewton}). In the next several subsections, we discuss how to choose the penalty parameter $\beta_k$ and step size $\alpha_k$, and appropriate courses of action for when the SP-BFGS curvature condition (\ref{eq:sp-bfgs-curv-cond}) fails.

\begin{algorithm}[h]
\caption{SP-BFGS Minimization Routine}
\label{alg:SP-BFGS-routine}
\begin{algorithmic}[1]
\Procedure{SP-BFGS-Minimize}{$f(x), g(x), x^0, H^0$}
\State $k \gets 0$
\State $H_k \gets H^0$ 
\State $x_k \gets x^0$ 
\While {Not Converged/Stopped}
  \State $p_k \gets - H_k g_k$
  \smallskip
  \State Choose step size $\alpha_k$
  \smallskip
  \State $x_{k+1} \gets x_{k} + \alpha_k p_k$
  \smallskip
  \State $s_{k} \gets x_{k+1} - x_{k}$ \qquad
  \smallskip
  \State $y_{k} \gets g_{k+1} - g_{k}$
  \smallskip
  \State Choose penalty parameter $\beta_k$
  \smallskip
  \If {$s_k^T y_k > - \frac{1}{\beta_k}$}
  \smallskip
  \State $\gamma_k \gets \frac{1}{(s_k^T y_k + \frac{1}{\beta_k})}, \quad \omega_k \gets \frac{1}{(s_k^T y_k + \frac{2}{\beta_k})}$
  \begin{equation*}
  \resizebox{.8\hsize}{!}{
  $H_{k+1} = \bigg ( I - \omega_k s_k y_k^T \bigg ) H_k \bigg ( I - \omega_k y_k s_k^T \bigg ) + \omega_k \bigg [ \frac{\gamma_k}{\omega_k} + (\gamma_k - \omega_k) y_k^T H_k y_k \bigg ] s_k s_k^T$}
  \end{equation*}
  \Else
	\State Trigger SP-BFGS curvature condition failure recovery procedure
    \smallskip
  \EndIf
  \State $k \gets k+1$
\EndWhile
\EndProcedure
\end{algorithmic}
\end{algorithm}

\subsection{Choosing The Penalty Parameter $\beta_k$}
\label{sec:choice-of-penalty-parameter}
As the choice of $\beta_k$ determines how strongly to bias the estimate of $H_{k+1}$ towards $H_{k}$, the choice of $\beta_k$ is fundamentally connected to the amount of noise present in the measured gradients $g_{k+1}$ and $g_{k}$. In brief, if the amount of noise present in the measured gradients is large, $\beta_k$ should be small to avoid overfitting the noise, and if the amount of noise present in the measured gradients is small, $\beta_k$ should be large to avoid underfitting curvature information. To make this point more rigorous, we introduce the following assumption. 
\begin{assumption}[Uniform Gradient Noise Bound] \label{assump:gradient-noise-bound}
There exists a nonnegative constant $\bar{\epsilon}_g \geq 0$ such that 
\begin{equation} \label{eq:uniform-gradient-noise-bound}
\norm{g(x) - \nabla \phi(x)}_2 = \norm{e(x)}_2 \leq \bar{\epsilon}_g, \qquad \forall x \in \mathbb{R}^n .
\end{equation} 
\end{assumption}

As $\nabla \phi(x)$ is continuous, for each $k \geq 0$ we have
\begin{equation}
\norm{\lim_{\alpha_k \downarrow 0} \big [ \nabla \phi(x_{k} + \alpha_k p_{k}) \big ] - \nabla \phi(x_{k})}_2 = 0 .
\end{equation}
However, due to noise we \textit{cannot} in general guarantee
\begin{equation}
\norm{\lim_{\alpha_k \downarrow 0} \big [ g(x_{k} + \alpha_k p_{k}) \big ] - g(x_{k})}_2 = 0 .
\end{equation}
Using the continuity of $\nabla \phi(x)$, Assumption~\ref{assump:gradient-noise-bound}, and the triangle inequality, one can conclude that 
\begin{equation}
0 \leq \norm{\lim_{\alpha_k \downarrow 0} \big [ g(x_{k} + \alpha_k p_{k}) \big ] - g(x_{k})}_2 \leq 2 \bar{\epsilon}_g .
\end{equation}
As a result, it is now clear that in the presence of uniformly bounded gradient noise, sending the step size $\alpha_k$ to zero, and thus $s_k$ to zero, only bounds the difference of measured gradients within a ball with radius dependent on the gradient noise bound $\bar{\epsilon}_g$. 

As $g_{k+1}$ and $g_{k}$ can be decomposed into smooth and noise components, so can $s_k^T y_k$, giving
\begin{equation}
s_k^T y_k = s_k^T y_k^{smooth} + s_k^T y_k^{noise} = s_k^T [ \nabla \phi_{k+1} - \nabla \phi_{k} ] + s_k^T [ e_{k+1} - e_{k} ] .
\end{equation}
In conjunction with the Cauchy-Schwarz inequality, Assumption~\ref{assump:gradient-noise-bound} implies that 
\begin{equation}
- 2 \bar{\epsilon}_g \norm{s_k}_2 \leq s_k^T [ e_{k+1} - e_{k} ] \leq 2 \bar{\epsilon}_g \norm{s_k}_2
\end{equation}
and so we have the lower and upper bounds
\begin{equation} \label{eq:curv-with-noise-lower-upper-bound}
- 2 \bar{\epsilon}_g \norm{s_k}_2 + s_k^T y_k^{smooth} \leq s_k^T y_k \leq s_k^T y_k^{smooth} + 2 \bar{\epsilon}_g \norm{s_k}_2 .
\end{equation}
From (\ref{eq:curv-with-noise-lower-upper-bound}), it is clear that the bound on the effect of the noise grows linearly with $\norm{s_k}_2$. However, by using the average Hessian $\bar{G}_{k}$ from (\ref{eq:average-Hessian-G}) and applying Taylor's theorem to $\nabla \phi$, it is also clear that
\begin{equation}
s_k^T y_k^{smooth} = s_k^T \bar{G}_{k} s_k = O \big ( \norm{s_k}_2^2 \big ) 
\end{equation}
and so
\begin{equation}
s_k^T y_k = O \big ( \norm{s_k}_2^2 \big ) + O \big ( \norm{s_k}_2 \big )
\end{equation}
where the $O \big ( \norm{s_k}_2^2 \big )$ term is due to the true curvature of the smooth function $\phi$, and the $O \big ( \norm{s_k}_2 \big )$ term is due to noise. Thus, we have now illustrated an important general behaviour given Assumption~\ref{assump:gradient-noise-bound}. As $\norm{s_k}_2$ dominates $\norm{s_k}_2^2$ as $\norm{s_k}_2 \rightarrow 0$, the effects of noise can dominate the true curvature for small $s_k$. Conversely, as $\norm{s_k}_2^2$ dominates $\norm{s_k}_2$ as $\norm{s_k}_2 \rightarrow +\infty$, the true curvature can dominate the effects of noise for large $s_k$. 

Given the above analysis, a simple strategy for choosing $\beta_k$ is to make $\beta_k$ grow linearly with $\norm{s_k}_2$, such as
\begin{equation} \label{eq:linear-beta-choice}
\beta_k = N_s \norm{s_k}_2 
\end{equation}
where $N_s > 0$ is a slope parameter. As $\norm{s_k}_2 \rightarrow 0$, $H_{k+1} \rightarrow H_{k}$, which is desirable because the effects of noise likely dominate as $\norm{s_k}_2 \rightarrow 0$. Increasingly biasing the estimate of $H_{k+1}$ towards $H_{k}$ reduces how much $H_{k+1}$ can be corrupted by noise, and relaxes the SP-BFGS curvature condition (\ref{eq:sp-bfgs-curv-cond}), reducing the likelihood of needing to trigger a recovery procedure described in Section~\ref{sec:mod-curv-fail-recovery-procedure}. Also, as shown earlier, because $\nabla \phi$ is continuous, the true difference of gradients is guaranteed to go to zero as $s_k$ approaches zero. As a result, without noise present, it is natural that $H_{k+1} \rightarrow H_{k}$ as $s_k \rightarrow 0$. In the presence of noise, we wish for this behaviour to be preserved. Informally, one can intuitively think of wanting $H_{k}$ to behave as an approximate \textit{average} inverse Hessian, and the averaging should remove the corrupting effects of noise, leaving $H_{k}$ to behave as if no noise were present. Similarly, as $\norm{s_k}_2 \rightarrow +\infty$, $\beta_k \rightarrow +\infty$, and one recovers the BFGS update in the limit, which is desirable because the effects of noise are likely dominated by the true curvature as $\norm{s_k}_2 \rightarrow +\infty$. The slope parameter $N_s$ dictates how sensitive $\beta_k$ is to $\norm{s_k}_2$, and should be set proportional to the gradient noise level (i.e. $\bar{\epsilon}_g$). Intuitively, if the gradient noise level is low, $\beta_k$ should grow quickly with $\norm{s_k}_2$, as the effect of noise diminishes quickly, and vice versa. 

It may also be desirable to modify (\ref{eq:linear-beta-choice}) to
\begin{equation} \label{eq:linear-beta-choice-min-threshold}
\beta_k = \max \bigg \{ N_s \norm{s_k}_2 - N_o, 0 \bigg \}
\end{equation}
where $N_o > 0$ is an intercept parameter. The inclusion of $N_o$ allows one to stop updating $H_k$ if $\norm{s_k}_2$ is sufficiently small. For example, it may be desirable to stop updating $H_k$ when one is very close to a stationary point, as gradient measurements are likely heavily dominated by noise.

\subsection{Choosing The Step Size $\alpha_k$}
\label{sec:choic-of-step-size}
Classically, during BFGS updating $\alpha_k$ is chosen to satisfy the Armijo-Wolfe conditions. As function and gradient evaluations are not corrupted by noise in the classical BFGS setting, we can write the Armijo condition, also known as the sufficient decrease condition, as
\begin{equation} \label{eq:armijo-condition-no-noise}
\phi_{k+1} \leq \phi_k + c_1 \alpha_k \nabla \phi_k^T p_k 
\end{equation}
and the Wolfe condition, also known as the curvature condition, as 
\begin{equation} \label{eq:wolfe-condition-no-noise}
\nabla \phi_{k+1}^T p_k \geq c_2 \nabla \phi_k^T p_k 
\end{equation}
where $0 < c_1 < c_2 < 1$, with well known choices being $c_1 = 10^{-4}$ and $c_2 = 0.9$. Observe that by adding $\nabla \phi_k^T p_k$ to both sides of (\ref{eq:wolfe-condition-no-noise}) and multiplying by $\alpha_k$, (\ref{eq:wolfe-condition-no-noise}) becomes
\begin{equation} \label{eq:wolfe-condition-no-noise-bfgs-notation}
y_k^T s_k = [ \nabla \phi_{k+1} - \nabla \phi_k ]^T \alpha_k p_k \geq (c_2 - 1) \nabla \phi_k^T \alpha_k p_k . 
\end{equation}
If $p_k$ is a descent direction then $\nabla \phi_k^T p_k < 0$, and combined with $(c_2 - 1) < 0$ and $\alpha_k > 0$, one sees that (\ref{eq:wolfe-condition-no-noise-bfgs-notation}) implies
\begin{equation}
y_k^T s_k \geq (c_2 - 1) \nabla \phi_k^T s_k > 0 
\end{equation}
so (\ref{eq:wolfe-condition-no-noise}) effectively enforces (\ref{eq:curvature-condition}) when no gradient noise is present. 

In the presence of noisy gradients, we argue that in general it no longer makes sense to enforce the Wolfe condition (\ref{eq:wolfe-condition-no-noise}). In the presence of gradient noise, (\ref{eq:wolfe-condition-no-noise}) becomes
\begin{equation} \label{eq:wolfe-condition-noise}
[ \nabla \phi_{k+1} + e_{k+1} ]^T p_k \geq c_2 [ \nabla \phi_k + e_k ]^T p_k 
\end{equation}
which can behave erratically once the noise vectors $e_{k+1}$ and $e_k$ start to dominate the gradient of $\phi$. For example, the noise vectors $e_{k+1}$ and $e_k$ can cause both sides of (\ref{eq:wolfe-condition-noise}) to erratically change sign, in which case whether or not the Wolfe condition is satisfied can be governed by randomness more than anything else. 

We argue that because the SP-BFGS update allows one to relax the curvature condition based on the value of $\beta_k$ as shown in the SP-BFGS curvature condition (\ref{eq:sp-bfgs-curv-cond}), it is appropriate to drop the Wolfe condition entirely in the presence of gradient noise and instead employ only a version of the sufficient decrease condition when choosing $\alpha_k$. In the situation where gradient noise is present but function noise is not (i.e. $f(x) = \phi(x)$ in (\ref{eq:noisy-function-decomposition})), one can use a backtracking line search based on the sufficient decrease condition, which can guarantee convergence to a neighborhood of a stationary point of $\phi$. The situation where noise is present in both function and gradient evaluations is trickier. Similar to the approach presented in Section 4.2 of \cite{DFONoisyFunctionsQuasiNewton}, one option is to use a backtracking line search with a relaxed sufficient decrease condition of the form
\begin{equation} \label{eq:armijo-condition-relaxed-noise}
f_{k+1} \leq f_k + c_1 \alpha_k g_k^T p_k + 2 \epsilon_A
\end{equation}
where $\epsilon_A \geq 0$ is a noise tolerance parameter and $p_k = - H_k g_k$. In Theorem 4.2 of \cite{DFONoisyFunctionsQuasiNewton}, the authors show that under Assumptions~\ref{assump:gradient-noise-bound} and~\ref{assump:function-noise-bound}, using the iteration (\ref{eq:line-search-iteration}) and a backtracking line search governed by the relaxed Armijo condition (\ref{eq:armijo-condition-relaxed-noise}) with $p_k = -g_k$ guarantees linear convergence to a neighborhood of the global minimizer for strongly convex functions. 

\begin{assumption}[Uniform Function Noise Bound] \label{assump:function-noise-bound}
There exists a nonnegative constant $\bar{\epsilon}_f \geq 0$ such that 
\begin{equation} \label{eq:uniform-function-noise-bound}
\abs{f(x) - \phi(x)} = \abs{\epsilon(x)} \leq \bar{\epsilon}_f, \qquad \forall x \in \mathbb{R}^n .
\end{equation} 
\end{assumption}

We agree with the authors of \cite{DFONoisyFunctionsQuasiNewton} that it is possible to prove an extension of Theorem 4.2 of \cite{DFONoisyFunctionsQuasiNewton} to a quasi-Newton iteration with positive definite $H_k$, and briefly outline why in Section~\ref{sec:fixed-alpha-convergence-analysis}. A quasi-Newton extension of Theorem 4.2 of \cite{DFONoisyFunctionsQuasiNewton} is relevant to SP-BFGS updating because, as we will formally see in Section~\ref{sec:beta-influence-on-H}, control of $\beta_k$ makes it possible to uniformly bound the minimum and maximum eigenvalues of $H_{k+1}$.

\subsection{Failed SP-BFGS Curvature Condition Recovery Procedure}
\label{sec:mod-curv-fail-recovery-procedure}
In the classical BFGS scenario where no gradient noise is present, the curvature condition (\ref{eq:curvature-condition}) may fail if $\alpha_k$ is not chosen based on the Armijo-Wolfe conditions and $\phi$ is not strongly convex. One of the most common strategies to handle this scenario is to skip the BFGS update (i.e. set $H_{k+1} = H_{k}$) when this occurs, which corresponds to an SP-BFGS update with $\beta_k = 0$. However, this simple strategy has the downside of potentially producing poor inverse Hessian approximations if updates are skipped too frequently. 

Conditionally skipping BFGS updates is an option in the presence of noisy gradients as well. In addition to skipping BFGS updates when (\ref{eq:curvature-condition}) fails, as described above, another course of action sometimes recommended in the presence of noise is to replace (\ref{eq:curvature-condition}) with
\begin{equation} \label{eq:curvature-condition-noise-modified-v1}
s_k^T y_k \geq \varepsilon \norm{s_k}_2^2
\end{equation}
where $\varepsilon > 0$ is a small positive constant, and skip the BFGS update if (\ref{eq:curvature-condition-noise-modified-v1}) is not satisfied. This strategy may be somewhat effective if $\norm{s_k}_2$ is large, but reduces back to the initial update skipping approach as $\norm{s_k}_2 \rightarrow 0$. A similar strategy (e.g. see Section 3.3.3 of \cite{DFONoisyFunctionsQuasiNewton}) is to replace (\ref{eq:curvature-condition}) with
\begin{equation} \label{eq:curvature-condition-noise-modified-v2}
s_k^T y_k \geq \zeta \norm{s_k}_2 \norm{y_k}_2
\end{equation}
and skip the BFGS update if (\ref{eq:curvature-condition-noise-modified-v2}) is not satisfied for some $\zeta \in (0,1)$. Notice that none of the aforementioned update skipping strategies allow for curvature information to be incorporated if the measured curvature $s_k^T y_k$ is negative. 

Unlike in the classical BFGS scenario, with SP-BFGS updating, curvature information can be incorporated even if the measured curvature $s_k^T y_k$ is negative by decreasing $\beta_k$ towards $0$. In addition to having the option of conditionally skipping updates (i.e. setting $\beta_k = 0$), one can also alternatively relax the SP-BFGS curvature condition by decreasing $\beta_k$ towards $0$ if (\ref{eq:sp-bfgs-curv-cond}) fails. Since $\beta_k$ is chosen after $s_k$ and $y_k$ are fixed in Algorithm~\ref{alg:SP-BFGS-routine}, one can solve for $\beta_k$ values satisfying (\ref{eq:sp-bfgs-curv-cond}) when $s_k^T y_k < 0$, yielding
\begin{equation} \label{eq:beta-curvature-fail-choice}
\beta_k = - \frac{1}{c_3 (s_k^T y_k)}
\end{equation}
for all $c_3 > 1$, assuming that $s_k \neq 0$ and $y_k \neq 0$. Note that if $s_k \neq 0$ and $y_k \neq 0$ and (\ref{eq:sp-bfgs-curv-cond}) fails, then the measured curvature $s_k^T y_k$ must be negative. The choice of $c_3$ determines how much to shrink $\beta_k$ compared to the largest value of $\beta_k$ that still satisfies (\ref{eq:sp-bfgs-curv-cond}) and thus guarantees the positive definiteness of $H_{k+1}$. Hence, if the value of $\beta_k$ produced by (\ref{eq:linear-beta-choice}) or (\ref{eq:linear-beta-choice-min-threshold}) is too large and (\ref{eq:sp-bfgs-curv-cond}) fails, one can choose an acceptable value of $\beta_k$ by using (\ref{eq:beta-curvature-fail-choice}) and selecting a $c_3 > 1$. Thus, instead of skipping the update (i.e. setting $\beta_k = 0$) if (\ref{eq:sp-bfgs-curv-cond}) fails, one can reduce $\beta_k$ towards $0$, which has the effect of reducing the magnitude of the update by increasing how much $H_{k+1}$ is biased towards $H_{k}$. An approach based on reducing $\beta_k$ towards $0$ never entirely skips incorporating measured curvature information, even if the measured curvature information is negative, but instead weights how heavily the measured curvature information affects $H_{k+1}$.

\section{Convergence of SP-BFGS}
\label{sec:convergence-analysis}
In this section, we discuss relevant theoretical and convergence properties of SP-BFGS. First, it is important to note that for specific choices of the sequence of penalty parameters $\beta_k$, known convergence results already exist. Specifically, if $\beta_k = +\infty$ for all $k$, then SP-BFGS updating is equivalent to BFGS updating. Although there are not many works on the convergence properties of BFGS updating in the presence of uniformly bounded noise, such as in Assumptions~\ref{assump:gradient-noise-bound} and~\ref{assump:function-noise-bound}, in \cite{doi:10.1137/19M1240794} the authors provide convergence results for a BFGS variant that employs an Armijo-Wolfe line search and lengthens the differencing interval in the presence of uniformly bounded function and gradient noise. At the other extreme, if $\beta_k = 0$ for all $k$, then one obtains a scaled gradient method for general $H^0 \succ 0$, and this becomes the gradient method when $H^0 = I$. Convergence analyses of the gradient method in the presence of uniformly bounded function and gradient noise for both a fixed step size and backtracking line search are provided in Section 4 of \cite{DFONoisyFunctionsQuasiNewton}. 

Given that perhaps the defining feature of SP-BFGS updating is the ability to vary $\beta_k$ at each iteration, we focus our attention on how varying $\beta_k$ can influence convergence behaviour in this section. As a result, most of the ensuing analysis centers around situations where the condition number of $H_k$ can be bounded. We do not employ the approach of bounding the cosine of the angle between the descent direction $p_k$ and the negative gradient above zero, and then showing that the condition number of $H_k$ is bounded, which is similar to the approaches taken when no noise is present in \cite{10.2307/2157646, 10.2307/2157680}, and when noise is present in \cite{doi:10.1137/19M1240794}. Although it may be possible to apply the strategies employed in \cite{10.2307/2157646, 10.2307/2157680, doi:10.1137/19M1240794} to establish convergence results for SP-BFGS, such an analysis is complicated enough that it is beyond the scope of this initial paper.

\subsection{The Influence of $\beta_k$ on $H_{k+1}$}
\label{sec:beta-influence-on-H}
We first examine how $\beta_k$ determines how much the maximum and minimum eigenvalues $\lambda_{max}(H_{k+1})$ and $\lambda_{min}(H_{k+1})$ can change. In what follows, $\lambda(H)$ denotes the set of eigenvalues $\lambda_1, \dots, \lambda_n$ of the matrix $H \in \mathbb{R}^{n \times n}$. We provide upper bounds on $\lambda_{max}(B_{k+1})$ and $\lambda_{max}(H_{k+1})$ via Theorem~\ref{thm:eigenvalue-bounds}. As $H_k = B_k^{-1}$, $1/\lambda_{min}(H_{k+1}) = \lambda_{max}(B_{k+1})$, and putting an upper bound on $\lambda_{max}(B_{k+1})$ is equivalent to putting a lower bound on $\lambda_{min}(H_{k+1})$. 
\begin{theorem}[Eigenvalue Upper Bounds] \label{thm:eigenvalue-bounds}
When $H_{k+1}$ is given by the SP-BFGS update (\ref{eq:SP-BFGS-Direct-Update-factored}), the following upper bounds (\ref{eq:H_k-upper-bound}) and (\ref{eq:B_k-upper-bound}) hold
\begin{equation} \label{eq:H_k-upper-bound}
\lambda_{max}(H_{k+1}) \leq \Tr(H_{k+1}) \leq \bigg [ \big ( 1 + \gamma_k \norm{y_k}_2 \norm{s_k}_2 \big )^2 \bigg ] \Tr(H_k) + \gamma_k \norm{s_k}_2^2 ,
\end{equation} 
\begin{equation} \label{eq:B_k-upper-bound}
\lambda_{max}(B_{k+1}) \leq \Tr(B_{k+1}) \leq \bigg [ 1 + \beta_k \norm{y_k}_2 \norm{s_k}_2 \bigg ] \Tr(B_k) + \gamma_k \norm{y_k}_2^2 .
\end{equation}
\end{theorem}
\begin{proof}
See Appendix~\ref{app:eigenvalue-bounds-proof}.
\smartqed
\end{proof}

With Theorem~\ref{thm:eigenvalue-bounds} in hand, we now formally see that when $s_k^T y_k > 0$, as $\beta_k$ increases from $0$ to $+\infty$, an upper bound on $\lambda_{max}(H_{k+1})$ interpolates between $\Tr(H_{k})$ and $+\infty$, and an upper bound on $\lambda_{max}(B_{k+1})$ interpolates between $\Tr(B_{k})$ and $+\infty$. Similarly, when $s_k^T y_k < 0$, as $\beta_k$ increases from $0$ towards $- \frac{1}{(s_k^T y_k)}$, an upper bound on $\lambda_{max}(H_{k+1})$ interpolates between between $\Tr(H_{k})$ and $+\infty$, and an upper bound on $\lambda_{max}(B_{k+1})$ interpolates between $\Tr(B_{k})$ and $+\infty$. Standard BFGS updating corresponds to setting $\beta_k = +\infty$ for all $k$, and as this is the largest possible value of $\beta_k$, one can no longer formally guarantee that $\lambda_{max}(B_{k+1})$ and $\lambda_{max}(H_{k+1})$ are bounded from above at each iteration because the measured curvature $s_k^T y_k$ may become arbitrarily close to zero due to the effects of noise. The key takeaway is that upper bounds on $\lambda_{max}(H_{k+1})$ and $\lambda_{max}(B_{k+1})$ can be tightened arbitrarily close to $\Tr(H_k)$ and $\Tr(B_k)$ by shrinking $\beta_k$ towards zero, as $s_k$, $y_k$, and $H_k$ are fixed before the value of $\beta_k$ is chosen in Algorithm~\ref{alg:SP-BFGS-routine}. 

Thus, if one must enforce a bound of the form $\lambda_{max}(H_{k+1}) \leq C_{H}$ or a bound of the form $\lambda_{max}(B_{k+1}) \leq C_{B}$ for all $k \geq 0$, where $C_{H} > \Tr(H_0) > 0$ and $C_{B} > \Tr(B_0) > 0$ are positive constants, there exist nontrivial sequences of sufficiently small $\beta_k$ with $\lim_{k \rightarrow \infty} \beta_k = 0$ that ensure the bounds hold for all $k$. To see this, observe that the interval $[\lambda_{max}(H_0), C_{H}]$ can be partitioned into subintervals corresponding to each iteration, and the sum of the subintervals cannot exceed $C_{H} - \lambda_{max}(H_0)$, which can be guaranteed by assigning a small enough value of $\beta_k$ to each subinterval, as this guarantees the maximum eigenvalue does not grow too much at each iteration $k$. Furthermore, although there clearly exist sequences of $\beta_k$ that ensure the bounds hold for all $k$ that satisfy $\beta_k = 0$ for all $k \geq K$, where $K$ is a positive integer, there also exist sequences of $\beta_k$ that ensure the bounds hold for all $k$ where $\beta_k$ instead only approaches zero in the limit $k \rightarrow \infty$.

\subsection{Minimization Of Strongly Convex Functions}
\label{sec:fixed-alpha-convergence-analysis}
Having established that SP-BFGS iterations can maintain bounds on the maximum and minimum eigenvalues of the approximate inverse Hessians via sufficiently small choices of $\beta_k$, we now consider minimizing strongly convex functions in the presence of bounded noise. We introduce Assumption~\ref{assump:strong-convexity}, and the notation $x^{\star}$ to denote the argument of the unique minimum of $\phi$, and  $\phi^{\star} = \phi(x^{\star})$ to denote the minimum. 
\begin{assumption}[Strong Convexity of $\phi$] \label{assump:strong-convexity}
The function $\phi \in C^2$ is twice continuously differentiable and there exist positive constants $0 < m \leq M$ such that 
\begin{equation} \label{eq:strongly-convex-hessian}
m I \preceq \nabla^2 \phi(x) \preceq M I, \qquad \forall x \in \mathbb{R}^n .
\end{equation}
\end{assumption}

We also state a general result in Lemma~\ref{lemma:noise-dominated-region} that establishes a region where $H_k g_k$ may not provide a descent direction with respect to $\phi$ due to noise dominating gradient measurements. Outside of this region, $H_k g_k$ is guaranteed to provide a descent direction for $\phi$. 
\begin{lemma}[Region Where Gradient Noise Can Dominate $\nabla \phi$] \label{lemma:noise-dominated-region}
Suppose Assumptions~\ref{assump:gradient-noise-bound} and~\ref{assump:strong-convexity}, and the decomposition in (\ref{eq:noisy-gradient-decomposition}) apply. Let $H$ be a symmetric positive definite matrix bounded by $\psi I \preceq H \preceq \Psi I$, where $0 < \psi \leq \Psi$. Define the neighborhood $\mathcal{N}_{1}(\psi,\Psi)$ as 
\begin{equation} \label{eq:noise-region-function}
\mathcal{N}_{1}(\psi,\Psi) \equiv \bigg \{ x \text{   } \big | \text{   } \phi(x) \leq \phi^{\star} + \frac{1}{2 m} \bigg ( \frac{\Psi \bar{\epsilon}_g}{\psi} \bigg )^2 \bigg \} .
\end{equation}
For all $x \notin \mathcal{N}_{1}$, $\nabla \phi(x)^T H g(x) > 0$. Contrapositively, for all $x$ such that $\nabla \phi(x)^T H g(x) \leq 0$, $x \in \mathcal{N}_{1}$. 
\end{lemma}
\begin{proof}
See Appendix~\ref{app:noise-dominated-region-proof}.
\smartqed
\end{proof}

Applying Lemma~\ref{lemma:noise-dominated-region} in the context of SP-BFGS updating makes several convergence properties clear. First, if one chooses $\beta_k$ such that $\psi I \preceq H_k \preceq \Psi I$ for all $k$ (i.e. the eigenvalues of the approximate inverse Hessian are uniformly bounded from above and below for all $k$), by Lemma~\ref{lemma:noise-dominated-region} it becomes clear that in the presence of gradient noise and absence of function noise (i.e. $f(x) = \phi(x)$ in (\ref{eq:noisy-function-decomposition})), the iterates of SP-BFGS with a backtracking line search based on (\ref{eq:armijo-condition-relaxed-noise}) with $\epsilon_{A} = 0$ in the worst case approach $\mathcal{N}_{1}$ as $k \rightarrow \infty$. To see this, observe that $H_k g_k$ is guaranteed to provide a descent direction outside of $\mathcal{N}_{1}$ and the sufficient decrease condition guarantees that $\alpha_k$ is not too large, while backtracking guarantees that $\alpha_k$ is not too small. For more background, see Chapter 3 of \cite{Nocedal2006}. 

Second, if both function and gradient noise are present, and one again chooses $\beta_k$ such that the bounds $\psi I \preceq H_k \preceq \Psi I$ hold for all $k$, under additional conditions a worst case analysis in Theorem~\ref{thm:fixed-alpha-linear-convergence} shows that an approach using a sufficiently small fixed step size $\alpha$ approaches $\mathcal{N}_{1}$ at a linear rate as $k \rightarrow \infty$. For a general quasi-Newton iteration of the form
\begin{equation} \label{eq:fixed-alpha-iteration}
x_{k+1} = x_{k} - \alpha H_k g_k 
\end{equation}
with constant step size $\alpha$ and $H_k \succ 0$, Theorem~\ref{thm:fixed-alpha-linear-convergence} establishes linear convergence to the region where noise can dominate $\nabla \phi$ (i.e. $\mathcal{N}_{1}$ in Lemma~\ref{lemma:noise-dominated-region}).
\begin{theorem}[Linear Convergence For Sufficiently Small Fixed $\alpha$] \label{thm:fixed-alpha-linear-convergence}
Suppose that Assumptions~\ref{assump:gradient-noise-bound} and~\ref{assump:strong-convexity} hold. Further suppose that $H_k$ is symmetric positive definite and bounded by $\psi I \preceq H_k \preceq \Psi I$, where $0 < \psi \leq \Psi$. Let $\psi$ be such that the inequality
\begin{equation} \label{eq:inner-product-proportional}
\nabla \phi_k^T H_k g_k \geq \psi \nabla \phi_k^T g_k
\end{equation}
is true for all $k$.
Let $\{ x_k \}$ be the iterates generated by~(\ref{eq:fixed-alpha-iteration}), where the constant step size $\alpha$ satisfies
\begin{equation} \label{eq:alpha-fixed-size-lower-upper-bound}
0 < \alpha \leq \frac{\psi}{\Psi^2 M} .
\end{equation}
Then for all $k$ such that $x_k \notin \mathcal{N}_{1}(\psi, \Psi)$, one has the Q-linear convergence result
\begin{equation} \label{eq:fixed-alpha-Q-linear-rate}
\phi_{k+1} - \bigg [ \phi^{\star} + \frac{1}{2 m} \bigg ( \frac{\Psi \bar{\epsilon}_g}{\psi} \bigg )^2 \bigg ] \leq (1 - \alpha \psi m) \bigg ( \phi_k - \bigg [ \phi^{\star} + \frac{1}{2 m} \bigg ( \frac{\Psi \bar{\epsilon}_g}{\psi} \bigg )^2 \bigg ] \bigg ) .
\end{equation}
Similarly, for any $x_0 \notin \mathcal{N}_{1}(\psi,\Psi)$, one has the R-linear convergence result
\begin{equation} \label{eq:fixed-alpha-R-linear-rate}
\phi_{k+1} - \phi^{\star} \leq (1 - \alpha \psi m)^k \bigg ( \phi_0 - \bigg [ \phi^{\star} + \frac{1}{2 m} \bigg ( \frac{\Psi \bar{\epsilon}_g}{\psi} \bigg )^2 \bigg ] \bigg ) + \frac{1}{2 m} \bigg ( \frac{\Psi \bar{\epsilon}_g}{\psi} \bigg )^2 . 
\end{equation}
\end{theorem}
\begin{proof}
See Appendix~\ref{app:fixed-alpha-linear-convergence-proof}.
\smartqed
\end{proof}

Theorem~\ref{thm:fixed-alpha-linear-convergence} can be considered a quasi-Newton extension of Theorem 4.1 from \cite{DFONoisyFunctionsQuasiNewton}, which lays the foundation for Theorem 4.2 from \cite{DFONoisyFunctionsQuasiNewton}. To extend the convergence result of Theorem~\ref{thm:fixed-alpha-linear-convergence} to the backtracking line search approach based on (\ref{eq:armijo-condition-relaxed-noise}), see that (\ref{eq:alpha-fixed-size-lower-upper-bound}), (\ref{eq:simplified-worst-case-descent-bound}), and Assumption~\ref{assump:function-noise-bound} combined imply 
\begin{equation} \label{eq:armijo-condition-relaxed-noise-theorem-4-bounded}
f(x_k - \alpha H_k g_k) \leq f(x_k) - \frac{\alpha \psi}{2} \bigg ( \norm{\nabla \phi_k}_2^2 - \norm{e_k}_2^2 \bigg ) + 2 \bar{\epsilon}_f
\end{equation}
and so if $\epsilon_A > \bar{\epsilon}_f$, comparing (\ref{eq:armijo-condition-relaxed-noise}) and (\ref{eq:armijo-condition-relaxed-noise-theorem-4-bounded}) makes it clear that (\ref{eq:armijo-condition-relaxed-noise}) will be satisfied for sufficiently small $\alpha$. Hence, the backtracking line search always finds an $\alpha_k$ satisfying (\ref{eq:armijo-condition-relaxed-noise}). For brevity, we defer a full, rigorous quasi-Newton extension of Theorem 4.2 of \cite{DFONoisyFunctionsQuasiNewton} to future work and instead investigate the performance of an approach based on (\ref{eq:armijo-condition-relaxed-noise}) via numerical experiments in Section~\ref{sec:numerical-experiments}.

\section{Numerical Experiments}
\label{sec:numerical-experiments}
In this section, we test instances of Algorithm~\ref{alg:SP-BFGS-routine} on a diverse set of 33 test problems for unconstrained minimization. The set of test problems includes convex and nonconvex functions, and well known pathological functions such as the Rosenbrock function \cite{10.1093/comjnl/3.3.175} and its relatives. Described in Section~\ref{sec:ill-conditioned-quadratic-test}, the first test problem is similar to the one used in the numerical experiments section of \cite{doi:10.1137/19M1240794}, and involves an ill conditioned quadratic function. The other 32 problems are selected problems from the CUTEst test problem set \cite{gould-orban-cutest}, and are used for tests in Section~\ref{sec:CUTEst-tests}. Code for running these numerical experiments was written in the Julia programming language \cite{doi:10.1137/141000671}, and utilizes the NLPModels.jl \cite{orban-siqueira-nlpmodels-2020}, CUTEst.jl \cite{orban-siqueira-cutest-2020}, and Distributions.jl \cite{2019arXiv190708611B, Distributions.jl-2019} packages. In all the numerical experiments that follow, noise $\epsilon(x)$ was added to function evaluations by uniformly sampling from the interval $[-\bar{\epsilon}_f,\bar{\epsilon}_f]$, and noise $e(x)$ was added to the gradient evaluations by uniformly sampling from the closed Euclidean ball $\norm{x}_2 \leq \bar{\epsilon}_g$.

\subsection{Ill Conditioned Quadratic Function With Additive Gradient Noise Only}
\label{sec:ill-conditioned-quadratic-test}
The first test problem is strongly convex and consists of the $4$-dimensional quadratic function given by
\begin{equation} \label{eq:quadratic-test-problem-smooth}
\phi(x) = \frac{1}{2} x^T T x
\end{equation}
where the eigenvalues of $T$ are $\lambda(T) = \{ 10^{-2}, 1, 10^{2}, 10^{4} \}$. Consequently, the strong convexity parameter is $m = 10^{-2}$, the Lipschitz constant is $M = 10^{4}$, and the condition number of the Hessian $T$ is $10^{6}$. For this test problem, no noise was added to the function evaluations (i.e. $f(x) = \phi(x)$ in (\ref{eq:noisy-function-decomposition})), and $\bar{\epsilon}_g = 1$. As a result, in this scenario $\mathcal{N}_1$ from Lemma~\ref{lemma:noise-dominated-region} with $\psi = \Psi = 1$ (i.e. the smallest possible $\mathcal{N}_1$) becomes
\begin{equation}
\mathcal{N}_1(1,1) = \bigg \{ x \text{   } \big | \text{   } \phi(x) \leq 50 \bigg \} .
\end{equation}

Following the discussion in Section~\ref{sec:choice-of-penalty-parameter}, we set the penalty parameters via the formula $\beta_k = \frac{1}{\bar{\epsilon}_g} \norm{s_k}_2 + 10^{-10}$, which corresponds to a choice of $N_s = 1$ in (\ref{eq:linear-beta-choice}). The $10^{-10}$ term was added as a small perturbation to provide numerical stability. The step size $\alpha_k$ was chosen using a backtracking line search based on the sufficient decrease condition (\ref{eq:armijo-condition-relaxed-noise}) with $p_k = - H_k g_k$, where $g_k$ is defined by (\ref{eq:noisy-gradient-decomposition}), $\epsilon_{A} = 0$, and $c_1 = 10^{-4}$. At each iteration, backtracking started from the initial step size $\alpha^{0} = 1$, decreasing by a factor of $\tau = 1/2$ each time the sufficient decrease condition failed. If the backtracking line search exceeded the maximum number of $75$ backtracks, we set $\alpha_k = 0$. However, the maximum number of backtracks was never exceeded when performing experiments with this first test problem. 

Algorithm~\ref{alg:SP-BFGS-routine} was initialized using the matrix and starting point
\begin{equation} \label{eq:initial-matrix-and-starting-point}
H^{0} = I, \qquad x^{0} = 10^{5} \cdot [1, 1, 1, 1]^T 
\end{equation}
given in (\ref{eq:initial-matrix-and-starting-point}), with $\norm{\nabla \phi(x^0)}_2 \approx 10^{9}$. Figures \ref{fig:optimality-gap}, \ref{fig:gradient-norm}, and \ref{fig:condition-number} compare the performance of $30$ independent runs of SP-BFGS vs. BFGS over a fixed budget of $100$ iterations. The relevant curvature condition failed an average of $25.7$ total iterations per BFGS run, and $0.6$ total iterations per SP-BFGS run. For the sake of comparability, both SP-BFGS and BFGS skipped the update if the relevant curvature condition failed. Observe that SP-BFGS reduces the objective function value by several more orders of magnitude compared to BFGS on average, and maintains significantly better inverse Hessian approximations than BFGS in the presence of gradient noise. 

\newpage
\begin{figure}[H]%
\centering
\begin{minipage}{\textwidth}
\centering
\begin{tabular}{cc}
\includegraphics[width=5.7cm]{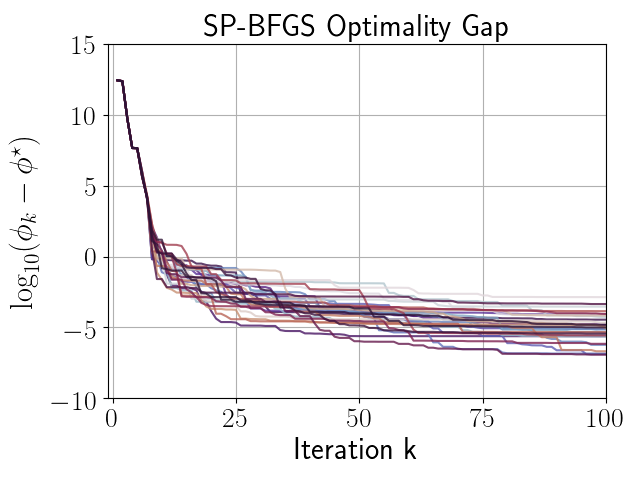} &
\includegraphics[width=5.7cm]{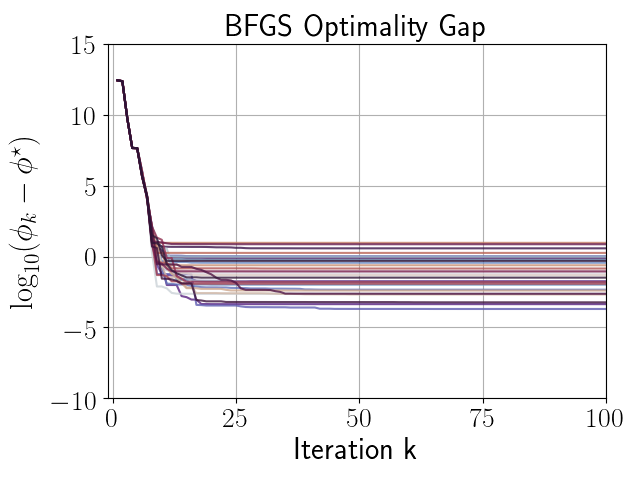} \\
\end{tabular}
\caption{Base 10 logarithm of the optimality gap vs. the iteration number $k$ for $30$ independent runs. After $100$ iterations, SP-BFGS has an average $\log_{10}(\phi_{100} - \phi^{\star})$ of $-5.03$ while BFGS has an average $\log_{10}(\phi_{100} - \phi^{\star})$ of $-1.27$. Observe that both SP-BFGS and BFGS appear to enter $\mathcal{N}_1(1,1)$, which corresponds to values less than $\log_{10}(50) \approx 1.7$ on the y-axis, but SP-BFGS makes more progress inside $\mathcal{N}_1(1,1)$. Outside of $\mathcal{N}_1(1,1)$, the performance of SP-BFGS and BFGS is almost indistinguishable. }
\label{fig:optimality-gap}
\end{minipage}%
\\
\bigskip
\begin{minipage}{\textwidth}
\centering
\begin{tabular}{cc}
\includegraphics[width=5.7cm]{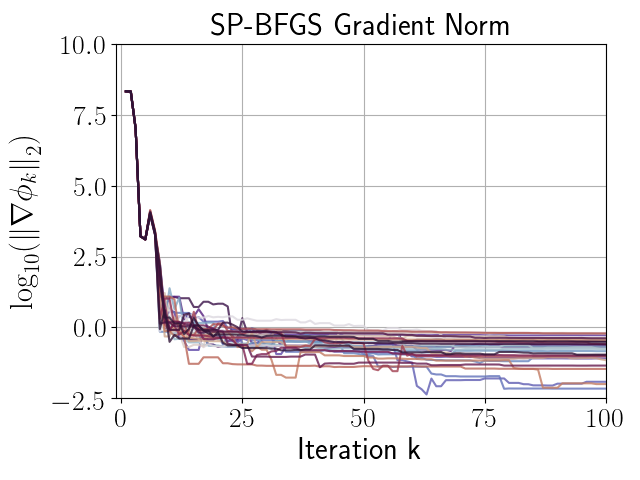} &
\includegraphics[width=5.7cm]{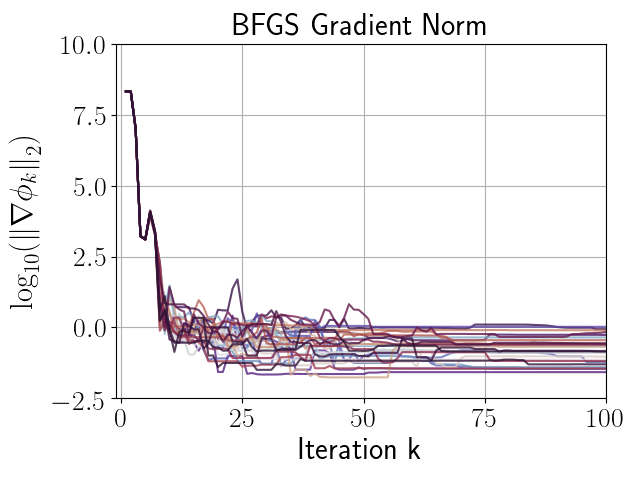} \\
\end{tabular}
\caption{Base 10 logarithm of the Euclidean norm of the true gradient $\nabla \phi_k$ vs. the iteration number $k$ for $30$ independent runs. Note that the BFGS values appear to vary more wildly than the SP-BFGS values. }
\label{fig:gradient-norm}
\end{minipage}%
\\
\bigskip
\begin{minipage}{\textwidth}
\centering
\begin{tabular}{cc}
\includegraphics[width=5.7cm]{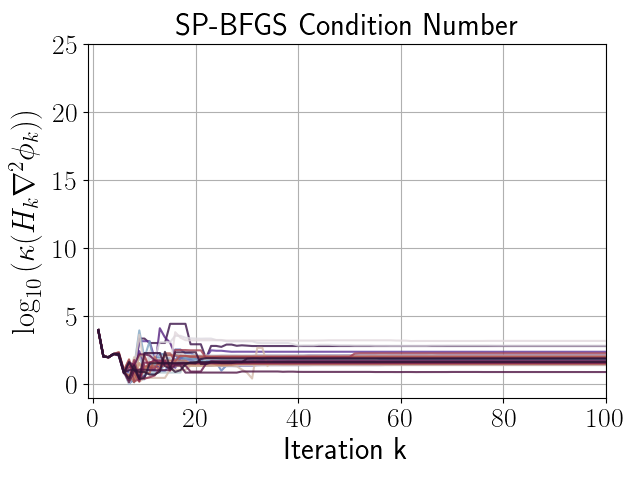} &
\includegraphics[width=5.7cm]{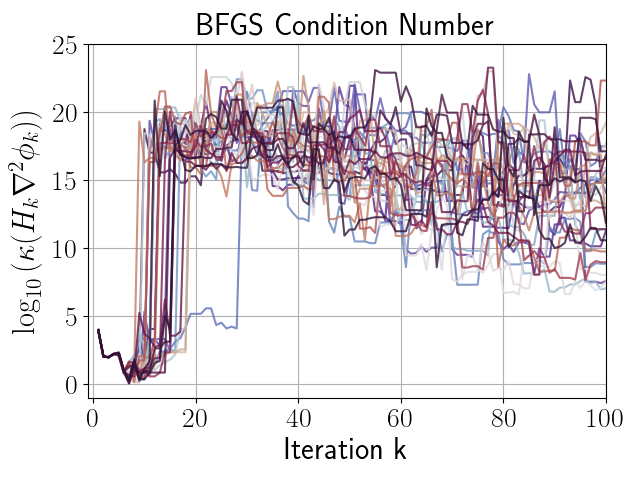} \\
\end{tabular}
\caption{Base 10 logarithm of the condition number of the true Hessian $\nabla^2 \phi_k$ scaled by the approximate inverse Hessian $H_k$ at each iteration $k$ for $30$ independent runs. As ideally one wants $H_k \nabla^2 \phi_k = I$, which has a condition number of $1$, the ideal value on these plots is $\log_{10}(1) = 0$. Observe how the BFGS approximation deteriorates massively inside $\mathcal{N}_1(1,1)$, and how SP-BFGS avoids this massive deterioration. From examining the BFGS $H_k$, the authors were able to determine that in the region of deterioration, the values of the entries of $H_k$ are often smaller than $10^{-5}$. }
\label{fig:condition-number}
\end{minipage}%
\end{figure}

\subsection{CUTEst Test Problems With Various Additive Noise Combinations}
\label{sec:CUTEst-tests}
The remaining $32$ test problems were selected from the CUTEst problem set, the successor of CUTEr \cite{gould-orban-toint-cuter}. At the time of writing, SIF files and descriptions of all $32$ test problems can be found at \url{https://www.cuter.rl.ac.uk/Problems/mastsif.shtml}. As a brief summary, some of the problems can be interpreted as least squares type problems (e.g. ARGTRGLS), some of the problems are ill conditioned or singular type problems (e.g. BOXPOWER), some of the problems are well known nonlinear optimization test problems (e.g. ROSENBR) or extensions of them (e.g. ROSENBRTU, SROSENBR), and some of the problems come from real applications (e.g. COATING, HEART6LS, VIBRBEAM). As shown in Tables~\ref{tab:SPBFGS-alg-comp-func-and-grad-noise} and~\ref{tab:BFGS-alg-comp-func-and-grad-noise}, the selected CUTEst test problems vary in size from $2$-dimensional to $1000$-dimensional.

Using these $32$ CUTEst test problems and a fixed budget of $2000$ objective function evaluations (not $2000$ iterations) per test, we tested the performance of SP-BFGS compared to BFGS with various combinations of function and gradient noise levels $\bar{\epsilon}_f$ and $\bar{\epsilon}_g$. For all the experiments in Tables~\ref{tab:rosenbrock-fixed-function-evals-noise-level-comp-results},~\ref{tab:SPBFGS-alg-comp-func-and-grad-noise}, and~\ref{tab:BFGS-alg-comp-func-and-grad-noise}, as well as the additional experiments in Appendix~\ref{app:extended-numerical-experiments}, both SP-BFGS and BFGS skipped updating if the curvature condition failed. In Tables~\ref{tab:rosenbrock-fixed-function-evals-noise-level-comp-results},~\ref{tab:SPBFGS-alg-comp-func-and-grad-noise}, and~\ref{tab:SPBFGS-alg-comp-grad-noise-only}, the SPBFGS penalty parameter was set as $\beta_k = \frac{10^8}{\bar{\epsilon}_g} \norm{s_k}_2 + 10^{-10}$, as the authors heuristically discovered setting $N_s = \frac{10^8}{\bar{\epsilon}_g}$ works well in practice for a variety of problems. With regards to the backtracking line search based on (\ref{eq:armijo-condition-relaxed-noise}), we set $\alpha^{0} = 1$, $\epsilon_{A} = \bar{\epsilon}_f$, $c_1 =10^{-4}$ , $\tau = 1/2$, and the maximum number of backtracks as $45$. We define $\Delta_{opt} \coloneqq \log_{10}(\phi_{best} - \phi^{\star})$ as a measure of the optimality gap, and use $\phi_{best}$ to denote the smallest value of the true function $\phi$ measured at any point during an algorithm run. The true minimum values $\phi^{\star}$ for each CUTEst problem were obtained from the SIF file for each CUTEst problem. The sample variance (i.e. the variance with Bessel's correction) is denoted by $s^2(\cdot)$. 

Table~\ref{tab:rosenbrock-fixed-function-evals-noise-level-comp-results} compares the performance of SP-BFGS vs. BFGS on the Rosenbrock function (i.e. ROSENBR) corrupted by different combinations of function and gradient noise of varying orders of magnitude. Observe that SP-BFGS outperforms BFGS with respect to the mean and median optimality gap for every noise combination in Table~\ref{tab:rosenbrock-fixed-function-evals-noise-level-comp-results}, sometimes by several orders of magnitude. Tables~\ref{tab:SPBFGS-alg-comp-func-and-grad-noise} and~\ref{tab:BFGS-alg-comp-func-and-grad-noise} compare the performance of SP-BFGS vs. BFGS on the $32$ CUTEst test problems with both function and gradient noise present. Gradient noise was generated using $\bar{\epsilon}_g = 10^{-4} \norm{\nabla \phi(x^0)}_2$, and function noise was generated using $\bar{\epsilon}_f = 10^{-4} \abs{\phi(x^0)}$, both to ensure that noise does not initially dominate function or gradient evaluations. Note that as the noise in these numerical experiments is additive, the signal to noise ratio of gradient measurements decreases as a stationary point is approached. Overall, SP-BFGS outperforms BFGS on approximately $70 \%$ of the CUTEst problems with both function and gradient noise present, and performs at least as good as BFGS on approximately $90 \%$ of these problems. Referring to Appendix~\ref{app:extended-numerical-experiments}, with only gradient noise present, these percentages become $80 \%$ and $95 \%$ respectively.

\newpage
\begin{table}[H]
\begin{center}
  {\def\arraystretch{1.5}\tabcolsep=3pt
  \begin{tabular}{ | c | c | c | c | c | c | c | c | }
	\hline
    $\bar{\epsilon}_f$ & $\bar{\epsilon}_g$ & \textit{Mean($\Delta_{opt}$)} & \textit{Median($\Delta_{opt}$)} & \textit{Min($\Delta_{opt}$)} & \textit{Max($\Delta_{opt}$)} & \textit{$s^2(\Delta_{opt})$} & \textit{Mean($I$)}\\ 
    \hline       
        \multicolumn{8}{ | c | }{ \textbf{SPBFGS With No Function Noise} } \\
    \hline
    $0$ & $10^{-4}$ & -1.4E+01 & -1.4E+01 & -1.8E+01 & -1.2E+01 & 1.4E+00 & 114 \\
    \hline
    $0$ & $10^{-2}$ & -1.3E+01 & -1.3E+01 & -1.5E+01 & -8.3E+00 & 2.9E+00 & 104 \\
    \hline
    $0$ & $10^{0}$ & -2.1E+00 & -1.8E+00 & -5.7E+00 & -9.2E-01 & 9.4E-01 & 153 \\
    \hline
    $0$ & $10^{2}$ & 3.5E-02 & 2.9E-01 & -1.9E+00 & 7.9E-01 & 3.9E-01 & 90 \\
    \hline             
        \multicolumn{8}{ | c | }{ \textbf{BFGS With No Function Noise} } \\
    \hline
    $0$ & $10^{-4}$ & -1.1E+01 & -1.0E+01 & -1.4E+01 & -8.8E+00 & 1.8E+00 & 263 \\
    \hline
    $0$ & $10^{-2}$ & -6.6E+00 & -6.6E+00 & -9.6E+00 & -4.3E+00 & 1.6E+00 & 281 \\
    \hline
    $0$ & $10^{0}$ & -1.5E+00 & -1.2E+00 & -3.3E+00 & -5.4E-01 & 6.3E-01 & 279 \\
    \hline
    $0$ & $10^{2}$ & 1.1E-01 &  4.3E-01 & -2.4E+00 & 6.5E-01 & 4.7E-01 & 373 \\
    \hline       
        \multicolumn{8}{ | c | }{ \textbf{SPBFGS With Low Function Noise Level} } \\
    \hline
    $10^{-4}$ & $10^{-4}$ & -1.4E+01 & -1.4E+01 & -1.5E+01 & -1.3E+01 & 1.9E-01 & 1980 \\
    \hline
    $10^{-4}$ & $10^{-2}$ & -1.0E+01 & -1.0E+01 & -1.2E+01 & -8.0E+00 & 1.3E+00 & 1964 \\
    \hline
    $10^{-4}$ & $10^{0}$ & -2.1E+00 & -2.0E+00 & -3.6E+00 & -1.6E+00 & 2.0E-01 & 1759 \\
    \hline
    $10^{-4}$ & $10^{2}$ & 8.7E-02 & 3.1E-01 & -2.2E+00 & 9.1E-01 & 4.5E-01 & 1720 \\
    \hline         
        \multicolumn{8}{ | c | }{ \textbf{BFGS With Low Function Noise Level} } \\
    \hline
    $10^{-4}$ & $10^{-4}$ & -1.1E+01 & -1.2E+01 & -1.5E+01 & -8.7E+00 & 1.7E+00 & 1980 \\
    \hline
    $10^{-4}$ & $10^{-2}$ & -6.6E+00 & -6.5E+00 & -8.8E+00 & -4.7E+00 & 1.2E+00 & 1975 \\
    \hline
    $10^{-4}$ & $10^{0}$ & -1.2E+00 & -1.1E+00 & -1.8E+00 & -8.6E-01 & 5.9E-02 & 1936 \\
    \hline
    $10^{-4}$ & $10^{2}$ & 9.5E-02 & 5.1E-01 & -3.1E+00 & 9.2E-01 & 8.5E-01 & 1934 \\
    \hline       
        \multicolumn{8}{ | c | }{ \textbf{SPBFGS With Medium Function Noise Level} } \\
    \hline  
    $10^{-2}$ & $10^{-4}$ & -1.4E+01 & -1.4E+01 & -1.5E+01 & -1.3E+01 & 3.4E-01 & 1981 \\
    \hline
    $10^{-2}$ & $10^{-2}$ & -1.0E+01 & -1.0E+01 & -1.3E+01 & -7.5E+00 & 1.5E+00 & 1977 \\
    \hline
    $10^{-2}$ & $10^{0}$ & -3.4E+00 & -3.0E+00 & -7.5E+00 & -2.0E+00 & 1.7E+00 & 1934 \\
    \hline
    $10^{-2}$ & $10^{2}$ & -1.8E-01 & 1.7E-01 & -3.7E+00 & 7.4E-01 & 1.0E+00 & 1890 \\
    \hline       
        \multicolumn{8}{ | c | }{ \textbf{BFGS With Medium Function Noise Level} } \\
    \hline
    $10^{-2}$ & $10^{-4}$ & -1.1E+01 & -1.1E+01 & -1.4E+01 & -8.5E+00 & 1.4E+00 & 1981 \\
    \hline
    $10^{-2}$ & $10^{-2}$ & -6.7E+00 & -6.7E+00 & -1.0E+01 & -4.9E+00 & 1.7E+00 & 1979 \\
    \hline
    $10^{-2}$ & $10^{0}$ & -1.8E+00 & -1.5E+00 & -3.8E+00 & -9.1E-01 & 6.3E-01 & 1961 \\
    \hline
    $10^{-2}$ & $10^{2}$ & 1.4E-01 & 3.9E-01 & -2.3E+00 & 8.5E-01 & 6.1E-01 & 1953 \\
    \hline       
        \multicolumn{8}{ | c | }{ \textbf{SPBFGS With High Function Noise Level} } \\
    \hline    
    $10^{0}$ & $10^{-4}$ & -1.4E+01 & -1.4E+01 & -1.5E+01 & -1.3E+01 & 2.2E-01 & 1980 \\
    \hline
    $10^{0}$ & $10^{-2}$ & -1.0E+01 & -1.0E+01 & -1.2E+01 & -7.3E+00 & 9.6E-01 & 1978 \\
    \hline
    $10^{0}$ & $10^{0}$ & -3.1E+00 & -2.8E+00 & -5.1E+00 & -1.7E+00 & 8.9E-01 & 1969 \\
    \hline
    $10^{0}$ & $10^{2}$ & -2.2E-01 & 1.1E-02 & -1.9E+00 & 8.4E-01 & 7.6E-01 & 1943 \\
    \hline     
        \multicolumn{8}{ | c | }{ \textbf{BFGS With High Function Noise Level} } \\
    \hline 
    $10^{0}$ & $10^{-4}$ & -1.1E+01 & -1.1E+01 & -1.3E+01 & -9.0E+00 & 1.4E+00 & 1980 \\
    \hline
    $10^{0}$ & $10^{-2}$ & -6.7E+00 & -6.4E+00 & -9.1E+00 & -5.0E+00 & 1.5E+00 & 1980 \\
    \hline
    $10^{0}$ & $10^{0}$ & -1.8E+00 & -1.4E+00 & -5.3E+00 & -8.2E-01 & 1.1E+00 & 1973 \\
    \hline
    $10^{0}$ & $10^{2}$ & -2.9E-02 & 3.7E-01 & -2.1E+00 & 8.9E-01 & 7.9E-01 & 1965 \\
    \hline    
  \end{tabular}
  }
\end{center}
\caption{Performance of SP-BFGS vs. BFGS on the Rosenbrock function (i.e. ROSENBR) corrupted by noise. $\Delta_{opt} \coloneqq \log_{10}(\phi_{best} - \phi^{\star})$ measures the optimality gap, where $\phi_{best}$ denotes the smallest value of the true function $\phi$ measured at any point during an algorithm run. The number of objective function evaluations is fixed at $2000$, but the number of iterations $I$ can vary. Statistics are calculated from a sample of $30$ runs per algorithm.  }
\label{tab:rosenbrock-fixed-function-evals-noise-level-comp-results}
\end{table}

\newpage
\begin{table}[H]
\begin{center}
  {\def\arraystretch{1.5}\tabcolsep=3pt
  \begin{tabular}{ | c | c | c | c | c | c | c | }
    \hline
        \multicolumn{7}{ | c | }{ \textbf{SP-BFGS With Function And Gradient Noise}} \\
    \hline
    \textit{Problem} & \textit{Dim.} & \textit{Mean($\Delta_{opt}$)} & \textit{Median($\Delta_{opt}$)} & \textit{Min($\Delta_{opt}$)} & \textit{Max($\Delta_{opt}$)} & \textit{$s^2(\Delta_{opt})$} \\ 
    \hline
    ARGTRGLS & 200 & 4.5E-02 & 4.8E-02 & 1.7E-02 & 8.0E-02 & 2.5E-04 \\
    \hline
    ARWHEAD & 500 & -2.5E+00 & -2.5E+00 & -2.6E+00 & -2.5E+00 & 2.6E-04 \\
    \hline
    BEALE & 2 & -1.1E+01 & -1.1E+01 & -1.4E+01 & -9.8E+00 & 8.0E-01 \\
    \hline
    BOX3 & 3 & -7.1E+00 & -6.8E+00 & -8.9E+00 & -6.5E+00 & 6.2E-01 \\
    \hline
    BOXPOWER & 100 & -3.8E+00 & -3.8E+00 & -4.2E+00 & -3.5E+00 & 5.0E-02 \\
    \hline
	BROWNBS & 2 & -1.2E+00 & -7.4E-01 & -5.2E+00 & 2.0E+00 & 3.5E+00 \\
	\hline
	BROYDNBDLS & 50 & -6.2E+00 & -6.2E+00 & -6.4E+00 & -6.0E+00 & 6.9E-03 \\
	\hline
	CHAINWOO & 100 & 1.7E+00 & 1.8E+00 & 7.7E-03 & 2.1E+00 & 1.6E-01 \\
	\hline
	CHNROSNB & 50 & -4.2E+00 & -4.0E+00 & -5.5E+00 & -3.6E+00 & 3.8E-01 \\
	\hline
	COATING & 134 & -2.7E-02 & -1.2E-02 & -1.3E-01 & 9.6E-02 & 3.5E-03 \\
	\hline
	COOLHANSLS & 9 & -1.2E+00 & -1.1E+00 & -1.6E+00 & -8.7E-01 & 1.7E-02 \\
	\hline
	CUBE & 2 & -5.2E+00 & -4.7E+00 & -8.9E+00 & -3.1E+00 & 2.2E+00 \\
	\hline
	CYCLOOCFLS & 20 & -8.4E+00 & -8.5E+00 & -9.1E+00 & -6.9E+00 & 3.0E-01 \\
	\hline
	EXTROSNB & 10 & -5.2E+00 & -5.2E+00 & -5.2E+00 & -5.1E+00 & 1.3E-03 \\
	\hline
	FMINSRF2 & 64 & -8.7E+00 & -8.7E+00 & -8.7E+00 & -8.6E+00 & 2.6E-04 \\
	\hline
	GENHUMPS & 5 & 4.1E-02 & 2.4E-01 & -2.9E+00 & 7.8E-01 & 4.5E-01 \\
	\hline
	GENROSE & 5 & -9.4E+00 & -9.3E+00 & -9.9E+00 & -9.1E+00 & 5.6E-02 \\
	\hline
	HEART6LS & 6 & -3.5E-01 & 2.7E-01 & -2.0E+00 & 1.2E+00 & 1.5E+00 \\
	\hline
	HELIX & 3 & -6.1E+00 & -6.0E+00 & -7.4E+00 & -4.5E+00 & 5.0E-01 \\
	\hline
	MANCINO & 30 & -2.1E+00 & -2.1E+00 & -2.5E+00 & -1.9E+00 & 1.2E-02 \\
	\hline
	METHANB8LS & 31 & -3.8E+00 & -3.9E+00 & -4.2E+00 & -3.4E+00 & 3.6E-02 \\
	\hline
	MODBEALE & 200 & 1.1E+00 & 1.0E+00 & 4.7E-01 & 1.8E+00 & 1.8E-01 \\
	\hline
	NONDIA & 10 & -4.2E-03 & -4.3E-03 & -4.4E-03 & -3.2E-03 & 9.1E-08 \\
	\hline
	POWELLSG & 4 & -6.1E+00 & -6.0E+00 & -7.9E+00 & -4.6E+00 & 9.1E-01 \\
	\hline
	POWER & 10 & -3.9E+00 & -3.8E+00 & -4.9E+00 & -3.3E+00 & 1.9E-01 \\
	\hline
	ROSENBR & 2 & -8.6E+00 & -8.5E+00 & -1.1E+01 & -6.3E+00 & 1.8E+00 \\
	\hline
	ROSENBRTU & 2 & -1.8E+01 & -1.8E+01 & -2.0E+01 & -1.7E+01 & 4.0E-01 \\
	\hline
	SBRYBND & 500 & 3.9E+00 & 3.9E+00 & 3.9E+00 & 3.9E+00 & 9.2E-06 \\
	\hline
	SINEVAL & 2 & -1.4E+01 & -1.4E+01 & -1.5E+01 & -1.3E+01 & 3.7E-01 \\
	\hline
	SNAIL & 2 & -1.2E+01 & -1.2E+01 & -1.4E+01 & -1.1E+01 & 2.9E-01 \\
	\hline
	SROSENBR & 1000 & 5.0E-01 & 5.0E-01 & 2.9E-01 & 6.8E-01 & 8.6E-03 \\
	\hline
	VIBRBEAM & 8 & 1.5E+00 & 1.5E+00 & 1.2E+00 & 2.1E+00 & 2.6E-02 \\
	\hline
	\end{tabular}
  }
\end{center}
\caption{Performance of SP-BFGS on $32$ selected CUTEst test problems with noise added to both function and gradient evaluations. The number of objective function evaluations is fixed at $2000$. $\Delta_{opt} \coloneqq \log_{10}(\phi_{best} - \phi^{\star})$ measures the optimality gap, where $\phi_{best}$ denotes the smallest value of the true function $\phi$ measured at any point during an algorithm run. Statistics are calculated from a sample of $30$ runs per algorithm, and the \textit{Dim.} column gives the problem dimension. The SPBFGS penalty parameter was set as $\beta_k = \frac{10^8}{\bar{\epsilon}_g} \norm{s_k}_2 + 10^{-10}$. For each problem, function noise was generated using $\bar{\epsilon}_f = 10^{-4} \abs{\phi(x^0)}$, and gradient noise was generated using $\bar{\epsilon}_g = 10^{-4} \norm{\nabla \phi(x^0)}_2$, where the starting point $x^0$ varies by CUTEst problem. }
\label{tab:SPBFGS-alg-comp-func-and-grad-noise}
\end{table}

\newpage
\begin{table}[H]
\begin{center}
  {\def\arraystretch{1.5}\tabcolsep=3pt
  \begin{tabular}{ | c | c | c | c | c | c | c | }
    \hline
        \multicolumn{7}{ | c | }{ \textbf{BFGS With Function And Gradient Noise}} \\
    \hline
    \textit{Problem} & \textit{Dim.} & \textit{Mean($\Delta_{opt}$)} & \textit{Median($\Delta_{opt}$)} & \textit{Min($\Delta_{opt}$)} & \textit{Max($\Delta_{opt}$)} & \textit{$s^2(\Delta_{opt})$} \\ 
    \hline
    ARGTRGLS & 200 & 5.6E-02 & 5.5E-02 & 2.4E-02 & 8.4E-02 & 2.7E-04 \\
    \hline
    ARWHEAD & 500 & -2.5E+00 & -2.5E+00 & -2.6E+00 & -2.5E+00 & 4.1E-04 \\
    \hline
    BEALE & 2 & -7.7E+00 & -7.8E+00 & -9.7E+00 & -6.1E+00 & 7.1E-01 \\
    \hline
    BOX3 & 3 & -6.5E+00 & -6.5E+00 & -6.7E+00 & -6.4E+00 & 4.7E-03 \\
    \hline
    BOXPOWER & 100 & -3.7E+00 & -3.7E+00 & -4.2E+00 & -3.4E+00 & 3.4E-02 \\
    \hline
	BROWNBS & 2 & 6.8E-01 & 1.3E+00 & -3.2E+00 & 3.1E+00 & 2.9E+00 \\
	\hline
	BROYDNBDLS & 50 & -6.0E+00 & -6.0E+00 & -6.3E+00 & -5.7E+00 & 2.6E-02 \\
	\hline
	CHAINWOO & 100 & 1.7E+00 & 1.7E+00 & 1.2E+00 & 2.1E+00 & 5.9E-02 \\
	\hline
	CHNROSNB & 50 & -4.2E+00 & -4.1E+00 & -5.7E+00 & -3.4E+00 & 4.4E-01 \\
	\hline
	COATING & 134 & -3.7E-02 & -5.7E-02 & -1.6E-01 & 8.0E-02 & 4.1E-03 \\
	\hline
	COOLHANSLS & 9 & -1.0E+00 & -1.0E+00 & -2.0E+00 & -4.5E-01 & 7.2E-02 \\
	\hline
	CUBE & 2 & -1.6E+00 & -1.4E+00 & -3.6E+00 & -9.7E-01 & 4.1E-01 \\
	\hline
	CYCLOOCFLS & 20 & -7.2E+00 & -7.2E+00 & -9.1E+00 & -5.8E+00 & 8.7E-01 \\
	\hline
	EXTROSNB & 10 & -5.2E+00 & -5.2E+00 & -5.2E+00 & -5.1E+00 & 1.8E-03 \\
	\hline
	FMINSRF2 & 64 & -8.6E+00 & -8.7E+00 & -8.8E+00 & -8.2E+00 & 2.8E-02 \\
	\hline
	GENHUMPS & 5 & 1.2E-01 & 1.2E-01 & -1.2E+00 & 8.1E-01 & 2.3E-01 \\
	\hline
	GENROSE & 5 & -7.5E+00 & -7.6E+00 & -9.1E+00 & -6.2E+00 & 7.3E-01 \\
	\hline
	HEART6LS & 6 & 3.1E-01 & 6.1E-01 & -1.9E+00 & 1.2E+00 & 1.4E+00 \\
	\hline
	HELIX & 3 & -4.5E+00 & -4.7E+00 & -7.0E+00 & -2.7E+00 & 1.1E+00 \\
	\hline
	MANCINO & 30 & -1.6E+00 & -1.6E+00 & -1.8E+00 & -1.3E+00 & 1.3E-02 \\
	\hline
	METHANB8LS & 31 & -3.9E+00 & -3.8E+00 & -4.4E+00 & -3.6E+00 & 5.5E-02 \\
	\hline
	MODBEALE & 200 & 1.1E+00 & 1.1E+00 & 2.9E-01 & 1.8E+00 & 1.6E-01 \\
	\hline
	NONDIA & 10 & -3.7E-03 & -3.8E-03 & -4.4E-03 & -2.6E-03 & 3.1E-07 \\
	\hline
	POWELLSG & 4 & -5.2E+00 & -5.2E+00 & -7.6E+00 & -4.2E+00 & 7.1E-01 \\
	\hline
	POWER & 10 & -3.5E+00 & -3.5E+00 & -4.1E+00 & -2.9E+00 & 1.0E-01 \\
	\hline
	ROSENBR & 2 & -5.9E+00 & -5.5E+00 & -9.2E+00 & -4.5E+00 & 1.4E+00 \\
	\hline
	ROSENBRTU & 2 & -1.6E+01 & -1.6E+01 & -1.8E+01 & -1.4E+01 & 1.5E+00 \\
	\hline
	SBRYBND & 500 & 3.9E+00 & 3.9E+00 & 3.9E+00 & 3.9E+00 & 2.7E-05 \\
	\hline
	SINEVAL & 2 & -1.1E+01 & -1.1E+01 & -1.3E+01 & -8.9E+00 & 1.3E+00 \\
	\hline
	SNAIL & 2 & -9.4E+00 & -9.2E+00 & -1.2E+01 & -8.0E+00 & 7.2E-01 \\
	\hline
	SROSENBR & 1000 & 5.4E-01 & 5.4E-01 & 3.6E-01 & 7.8E-01 & 6.9E-03 \\
	\hline
	VIBRBEAM & 8 & 1.7E+00 & 1.7E+00 & 1.2E+00 & 2.0E+00 & 2.9E-02 \\
	\hline
	\end{tabular}
  }
\end{center}
\caption{Performance of BFGS on $32$ selected CUTEst test problems with noise added to both function and gradient evaluations. The number of objective function evaluations is fixed at $2000$. $\Delta_{opt} \coloneqq \log_{10}(\phi_{best} - \phi^{\star})$ measures the optimality gap, where $\phi_{best}$ denotes the smallest value of the true function $\phi$ measured at any point during an algorithm run. Statistics are calculated from a sample of $30$ runs per algorithm, and the \textit{Dim.} column gives the problem dimension. For each problem, function noise was generated using $\bar{\epsilon}_f = 10^{-4} \abs{\phi(x^0)}$, and gradient noise was generated using $\bar{\epsilon}_g = 10^{-4} \norm{\nabla \phi(x^0)}_2$, where the starting point $x^0$ varies by CUTEst problem. }
\label{tab:BFGS-alg-comp-func-and-grad-noise}
\end{table}

\section{Final Remarks}
\label{sec:final-remarks}
In this paper, we introduced SP-BFGS, a new variant of the BFGS method designed to resist the corrupting effects of noise. Motivated by regularized least squares estimation, we derived the SP-BFGS update by applying a penalty method to the secant condition. We argued that with an appropriate choice of penalty parameter, SP-BFGS updating is robust to the corrupting effects of noise that can destroy the performance of BFGS. We empirically validated this claim by performing numerical experiments on a diverse set of over $30$ test problems with both function and gradient noise of varying orders of magnitude. The results of these numerical experiments showed that SP-BFGS can outperform BFGS approximately $70 \%$ or more of the time, and performs at least as good as BFGS approximately $90 \%$ or more of the time. Furthermore, a theoretical analysis confirmed that with appropriate choices of penalty parameter, it is possible to guarantee that SP-BFGS is not corrupted arbitrarily badly by noise, unlike standard BFGS. In the future, we believe it is worth investigating the performance of SP-BFGS in the presence of other types of noise, including multiplicative stochastic noise and deterministic noise, and also believe it is worthwhile to study the use of noise estimation techniques in conjunction with SP-BFGS updating. The authors are also working to publish a limited memory version of SP-BFGS for high dimensional noisy problems.

\begin{acknowledgements}
EH and BI's work is supported by the Natural Sciences and Engineering Research Council of Canada (NSERC) and the University of British Columbia (UBC).
\end{acknowledgements}

%
\section*{Conflict of interest}

The authors declare that they have no conflict of interest.

\bibliographystyle{spmpsci}      
\bibliography{references}   

%
%

\newpage

\appendix 

\section{Proof of Theorem~\ref{thm:sp-bfgs-update}}
\label{app:SPBFGS-update-proof}
\normalsize
To produce the SP-BFGS update, we first rearrange (\ref{eq:SP-BFGS-lagrange-FOC}), revealing that
\begin{equation} \label{eq:FOC_1_rearranged}
(H - H_k) = - W^{-1} ( u y_k^T + \Gamma^T - \Gamma ) W^{-1}
\end{equation}
and so the symmetry requirement that $H = H^T$ means transposing (\ref{eq:FOC_1_rearranged}) gives
\begin{equation}
u y_k^T + \Gamma^T - \Gamma  =  ( u y_k^T + \Gamma^T - \Gamma )^T = y_k u^T + \Gamma - \Gamma^T
\end{equation}
which rearranges to
\begin{equation}
\Gamma^T - \Gamma = \frac{1}{2} ( y_k u^T - u y_k^T ) 
\end{equation}
and so
\begin{equation} \label{eq:eliminate-big-gamma}
(H - H_k) = - \frac{1}{2} W^{-1} (y_k u^T + u y_k^T) W^{-1}.
\end{equation}
Next, we right multiply (\ref{eq:eliminate-big-gamma}) by $y_k$ to get
\begin{equation}
(H - H_k) y_k = - \frac{1}{2} W^{-1} \bigg ( y_k u^T W^{-1} y_k + u (y_k^T W^{-1} y_k) \bigg )
\end{equation}
and use (\ref{eq:SP-BFGS-sec-cond}) to get that
\begin{equation}
s_k + \frac{W^{-1} u}{\beta_k} - H_k y_k = - \frac{1}{2} W^{-1} \bigg ( y_k u^T W^{-1} y_k + u (y_k^T W^{-1} y_k) \bigg ) .
\end{equation}
We now left multiply both sides by $-2 W$ and rearrange, giving
\begin{equation}
-2 W (s_k - H_k y_k) = y_k u^T W^{-1} y_k + u \bigg ( y_k^T W^{-1} y_k + \frac{2}{\beta_k} \bigg ).  
\end{equation}
This can be rearranged so that $u$ is isolated, giving
\small
\begin{equation} \label{eq:lambda-partial-isolation}
u = \frac{-2 W (s_k - H_k y_k) - y_k u^T W^{-1} y_k}{y_k^T W^{-1} y_k + \frac{2}{\beta_k}} = - \frac{2 W (s_k - H_k y_k) + y_k u^T W^{-1} y_k}{y_k^T W^{-1} y_k + \frac{2}{\beta_k}}.
\end{equation}
\normalsize
To get rid of the $u^T$ on the right hand side, we first left multiply both sides by $y_k^T W^{-1}$, and then transpose to get 
\begin{equation}
u^T W^{-1} y_k = - \frac{2 (s_k - H_k y_k)^T y_k + (y_k^T W^{-1} y_k) (u^T W^{-1} y_k)}{y_k^T W^{-1} y_k + \frac{2}{\beta_k}}
\end{equation}
where we have taken advantage of the fact that the transpose of a scalar returns the same scalar. This now allows us to solve for $u^T W^{-1} y_k$ using some basic algebra, and resulting in
\begin{equation} \label{eq:int-lambda-transpose-result}
u^T W^{-1} y_k = - \frac{(s_k - H_k y_k)^T y_k}{y_k^T W^{-1} y_k + \frac{1}{\beta_k}} .
\end{equation}
Substituting (\ref{eq:int-lambda-transpose-result}) into (\ref{eq:lambda-partial-isolation}) gives
\begin{equation} \label{eq:lambda-isolated}
u = \frac{y_k y_k^T (s_k - H_k y_k)}{(y_k^T W^{-1} y_k + \frac{2}{\beta_k})(y_k^T W^{-1} y_k + \frac{1}{\beta_k})} - \frac{2 W (s_k - H_k y_k)}{y_k^T W^{-1} y_k + \frac{2}{\beta_k}}.
\end{equation}
Now, if we substitute the expression for $u$ in (\ref{eq:lambda-isolated}) into (\ref{eq:eliminate-big-gamma}), after some simplification we get 
\small
\begin{multline}
(H - H_k) = \frac{1}{(y_k^T W^{-1} y_k + \frac{2}{\beta_k})} \bigg [ (s_k - H_k y_k) y_k^T W^{-1} + W^{-1} y_k (s_k - H_k y_k)^T \\ - \frac{y_k^T(s_k - H_k y_k)}{(y_k^T W^{-1} y_k + \frac{1}{\beta_k})} W^{-1} y_k y_k^T W^{-1} \bigg ].
\end{multline}
\normalsize
Now, we further simplify by applying that $W s_k = y_k$, and thus $W^{-1} y_k = s_k$, revealing
\small
\begin{equation}
H = H_k + \frac{(s_k - H_k y_k) s_k^T + s_k (s_k - H_k y_k)^T}{(y_k^T s_k + \frac{2}{\beta_k})} - \frac{y_k^T(s_k - H_k y_k)}{(y_k^T s_k + \frac{2}{\beta_k})(y_k^T s_k + \frac{1}{\beta_k})} s_k s_k^T 
\end{equation}
\normalsize
which, after a bit of algebra, reveals that the update formula solving the system defined by (\ref{eq:SP-BFGS-lagrange-FOC}),~(\ref{eq:SP-BFGS-sec-cond}), and~(\ref{eq:SP-BFGS-sym-cond}) can be expressed as
\begin{equation} \label{eq:sp-bfgs-expanded}
H^{*} = H_k - \frac{H_k y_k s_k^T + s_k y_k^T H_k^T}{(y_k^T s_k + \frac{2}{\beta_k})} + \bigg [ \frac{y_k^T s_k + \frac{2}{\beta_k} + y_k^T H_k y_k}{(y_k^T s_k + \frac{2}{\beta_k})(y_k^T s_k + \frac{1}{\beta_k})} \bigg ] s_k s_k^T .
\end{equation}
We can make (\ref{eq:sp-bfgs-expanded}) look similar to the common form of the BFGS update given in (\ref{eq:BFGS-Direct-Update-Rho}) by defining the two quantities $\gamma_k$ and $\omega_k$ as in (\ref{eq:gamma-omega-definitions}) and observing that completing the square gives
\begin{multline}
H^{*} = \bigg ( I - \frac{s_k y_k^T}{(y_k^T s_k + \frac{2}{\beta_k})} \bigg ) H_k \bigg ( I - \frac{y_k s_k^T}{(y_k^T s_k + \frac{2}{\beta_k})} \bigg ) \\ + \bigg [ \frac{y_k^T s_k + \frac{2}{\beta_k} + y_k^T H_k y_k}{(y_k^T s_k + \frac{2}{\beta_k})(y_k^T s_k + \frac{1}{\beta_k})} - \frac{y_k^T H_k y_k}{(y_k^T s_k + \frac{2}{\beta_k})^2} \bigg ] s_k s_k^T 
\end{multline}
which is equivalent to
\begin{equation} \label{eq:SP-BFGS-Direct-Update-factored-o}
H^{*} = \bigg ( I - \omega_k s_k y_k^T \bigg ) H_k \bigg ( I - \omega_k y_k s_k^T \bigg ) + \omega_k \bigg [ \frac{\gamma_k}{\omega_k} + (\gamma_k - \omega_k) y_k^T H_k y_k \bigg ] s_k s_k^T 
\end{equation}
concluding the proof.

\section{Proof of Lemma~\ref{thm:sp-bfgs-curv-cond}}
\label{app:sp-bfgs-curv-cond-proof}
The $H_{k+1}$ given by (\ref{eq:SP-BFGS-Direct-Update-factored}) has the general form 
\begin{equation} \label{eq:SP-BFGS-symmetric-decomposition}
H_{k+1} = G^T H_k G + d s_k s_k^T
\end{equation}
with the specific choices
\begin{equation} \label{eq:SP-BFGS-symmetric-decomposition-specific-choices}
G = I - \omega_k y_k s_k^T, \quad d = \omega_k \bigg [ \frac{\gamma_k}{\omega_k} + (\gamma_k - \omega_k) y_k^T H_k y_k \bigg ] .
\end{equation}
By definition, $H_{k+1}$ is positive definite if 
\begin{equation} \label{eq:positive-definite-definition}
v^T H_{k+1} v > 0, \quad \forall v \in \mathbb{R}^{n} \setminus 0  \text{~~}.
\end{equation}
We first show that (\ref{eq:sp-bfgs-curv-cond}) is a sufficient condition for $H_{k+1}$ to be positive definite, given that $H_k$ is positive definite. By applying (\ref{eq:SP-BFGS-symmetric-decomposition}) to (\ref{eq:positive-definite-definition}), we see that
\begin{equation} \label{eq:positive-definite-definition-decomp}
v^T \bigg ( G^T H_k G + d s_k s_k^T \bigg ) v > 0, \quad \forall v \in \mathbb{R}^{n} \setminus 0  
\end{equation}
must be true for the choices of $G$ and $d$ in (\ref{eq:SP-BFGS-symmetric-decomposition-specific-choices}) if $H_{k+1}$ is positive definite. Substituting (\ref{eq:SP-BFGS-symmetric-decomposition-specific-choices}) into (\ref{eq:positive-definite-definition-decomp}) reveals that 
\small
\begin{equation} \label{eq:positive-definite-definition-decomp-expanded}
\bigg ( v - \omega_k (s_k^T v) y_k \bigg )^T H_k \bigg ( v - \omega_k (s_k^T v) y_k \bigg ) + \omega_k \bigg [ \frac{\gamma_k}{\omega_k} + (\gamma_k - \omega_k) y_k^T H_k y_k \bigg ] (s_k^T v)^2 > 0 
\end{equation}
\normalsize
must be true for all $v \in \mathbb{R}^{n} \setminus 0$ if $H_{k+1}$ is positive definite. Both $(s_k^T v)^2$ and $v^T G^T H_k G v$ are always nonnegative. To see that $v^T G^T H_k G v \geq 0$, note that because $H_k$ is positive definite, it has a principal square root $H_k^{1/2}$, and so 
\begin{equation}
v^T G^T H_k G v = v^T G^T H_k^{1/2} H_k^{1/2} G v = \norm{H_k^{1/2} G v}_2^2 \geq 0  \text{~} .  
\end{equation}
We now observe that if $d > 0$, the right term $d (s_k^T v)^2$ in (\ref{eq:positive-definite-definition-decomp-expanded}) is zero if and only if $(s_k^T v) = 0$. However, if $(s_k^T v) = 0$, then the left term $v^T G^T H_k G v$ in (\ref{eq:positive-definite-definition-decomp-expanded}) is zero only when $v = 0$. Hence, the condition $d > 0$ guarantees that (\ref{eq:positive-definite-definition-decomp-expanded}) is true for all $v$ excluding the zero vector, and thus that $H_{k+1}$ is positive definite. The condition $d > 0$ expands to
\begin{equation} \label{eq:d-positive-definite}
\gamma_k + \omega_k (\gamma_k - \omega_k) y_k^T H_k y_k > 0  \text{~} . 
\end{equation}
Using the definitions of $\gamma_k$ and $\omega_k$ in (\ref{eq:gamma-omega-definitions}), it is clear that $(\gamma_k - \omega_k) \geq 0$, as $\beta_k$ can only take nonnegative values. Furthermore, as $H_k$ is positive definite, $y_k^T H_k y_k \geq 0$ for all $y_k$. As it is possible for $(\gamma_k - \omega_k) y_k^T H_k y_k$ to be zero, we requre $\gamma_k > 0$. The condition $\gamma_k > 0$ immediately gives (\ref{eq:sp-bfgs-curv-cond}), as $\gamma_k$ can only be positive if the denominator in its definition is positive. Finally, as $\beta_k$ can only take nonnegative values, (\ref{eq:sp-bfgs-curv-cond}) also ensures that $\omega_k$ is nonnegative, and so when (\ref{eq:sp-bfgs-curv-cond}) is true, $\omega_k (\gamma_k - \omega_k) y_k^T H_k y_k \geq 0$. In summary, we have shown that the condition (\ref{eq:sp-bfgs-curv-cond}) ensures that the left term in (\ref{eq:d-positive-definite}) is positive, and the right term nonnegative, so $d > 0$, and thus $H_{k+1}$ is positive definite.

We now show that (\ref{eq:sp-bfgs-curv-cond}) is a necessary condition for $H_{k+1}$ to be positive definite, given that $H_k$ is positive definite. If $H_{k+1}$ is positive definite, then 
\begin{equation} \label{eq:H-pd}
y_k^T H_{k+1} y_k > 0 
\end{equation}
assuming $y_k \neq 0$. Substituting (\ref{eq:SP-BFGS-sec-cond}) into (\ref{eq:H-pd}) gives
\begin{equation} \label{eq:sp-bfgs-curv-cond-unsimplified}
y_k^T \bigg [ s_{k} + \frac{W^{-1} u}{\beta_k} \bigg ] > 0 
\end{equation}
and using (\ref{eq:int-lambda-transpose-result}) shows that~(\ref{eq:sp-bfgs-curv-cond-unsimplified}) is equivalent to
\begin{equation}
y_k^T \bigg [ s_{k} + \frac{\gamma_k (H_k y_k - s_k)}{\beta_k} \bigg ] > 0 .
\end{equation}
Now, some algebra shows that
\small
\begin{equation} \label{eq:convex-combination}
\begin{split}
y_k^T \bigg [ s_{k} + \frac{\gamma_k (H_k y_k - s_k)}{\beta_k} \bigg ] & = y_k^T s_{k} + \frac{1}{1 + \beta_k y_k^T s_{k}} \bigg [ y_k^T H_k y_k - y_k^T s_{k} \bigg ] \\
 & = \bigg ( 1 - \frac{1}{1 + \beta_k y_k^T s_{k}} \bigg ) y_k^T s_{k} + \bigg ( \frac{1}{1 + \beta_k y_k^T s_{k}} \bigg ) y_k^T H_k y_k \\
 & = \bigg ( \frac{\beta_k y_k^T s_{k}}{1 + \beta_k y_k^T s_{k}} \bigg ) y_k^T s_{k} + \bigg ( \frac{1}{1 + \beta_k y_k^T s_{k}} \bigg ) y_k^T H_k y_k \\
 & = \frac{\beta_k (y_k^T s_{k})^2 + y_k^T H_k y_k}{1 + \beta_k y_k^T s_{k}} 
\end{split}
\end{equation}
\normalsize
and we also know that because $H_k$ is positive definite, $y_k^T H_k y_k > 0$ for all $y_k \neq 0$, by definition $\beta_k \geq 0$, and by the definition of the square of a real number, $(y_k^T s_{k})^2 \geq 0$. As a result, 
\begin{equation}
y_k^T \bigg [ s_{k} + \frac{W^{-1} u}{\beta_k} \bigg ] = \frac{\beta_k (y_k^T s_{k})^2 + y_k^T H_k y_k}{1 + \beta_k y_k^T s_{k}} > 0 
\end{equation}
is guaranteed only if the denominator $1 + \beta_k y_k^T s_{k}$ is positive, which occurs when
\begin{equation}
s_k^T y_k > - \frac{1}{\beta_k} .
\end{equation}
This establishes that (\ref{eq:sp-bfgs-curv-cond}) is a necessary condition for $H_{k+1}$ to be positive definite, given that $H_k$ is positive definite, and concludes the proof.

\section{Proof of Theorem~\ref{thm:sp-bfgs-inverse-update}}
\label{app:sp-bfgs-inverse-update-proof}
The Sherman-Morrison-Woodbury formula says
\begin{equation} \label{eq:Sherman-Morrison-Woodbury-formula}
(A + UCV)^{-1} = A^{-1} - A^{-1} U (C^{-1} + V A^{-1} U)^{-1} V A^{-1} .
\end{equation}
Now, observe that the SP-BFGS update (\ref{eq:SP-BFGS-Direct-Update-factored}) can be written in the factored form
\begin{equation} \label{eq:SP-BFGS-Direct-Update-block-factored}
H_{k+1} = H_k + \omega_k \big [ s_k \quad H_k y_k \big ] \left[ \begin{array}{cc}
\gamma_k \big ( \frac{1}{\omega_k} + y_k^T H_k y_k \big ) & -1 \\
-1 & 0 \end{array} 
\right] \left[ \begin{array}{c}
s_k^T \\
y_k^T H_k \end{array} 
\right] .
\end{equation}
Applying the Sherman-Morrison-Woodbury formula (\ref{eq:Sherman-Morrison-Woodbury-formula}) to the factored SP-BFGS update (\ref{eq:SP-BFGS-Direct-Update-block-factored}) with
\begin{equation*}
\resizebox{.95\hsize}{!}{$A = H_k, \quad U = \omega_k \big [ s_k \quad H_k y_k \big ], \quad C = \left[ \begin{array}{cc}
\gamma_k \big ( \frac{1}{\omega_k} + y_k^T H_k y_k \big ) & -1 \\
-1 & 0 \end{array} 
\right], \quad V = \left[ \begin{array}{c}
s_k^T \\
y_k^T H_k \end{array} 
\right]$} 
\end{equation*}
yields 
\begin{equation*}
\resizebox{.99\hsize}{!}{$H_{k+1}^{-1} = H_k^{-1} - H_k^{-1} \omega_k \big [ s_k \quad H_k y_k \big ] \bigg ( \left[ \begin{array}{cc}
\gamma_k \big ( \frac{1}{\omega_k} + y_k^T H_k y_k \big ) & -1 \\
-1 & 0 \end{array} 
\right]^{-1} + \left[ \begin{array}{c}
s_k^T \\
y_k^T H_k \end{array} 
\right] H_k^{-1} \omega_k \big [ s_k \quad H_k y_k \big ] \bigg )^{-1} \left[ \begin{array}{c}
s_k^T \\
y_k^T H_k \end{array} 
\right] H_k^{-1} . $ }
\end{equation*}
Inverting $C$ here gives
\begin{equation*}
C^{-1} = \left[ \begin{array}{cc}
\gamma_k \big ( \frac{1}{\omega_k} + y_k^T H_k y_k \big ) & -1 \\
-1 & 0 \end{array} 
\right]^{-1} = \left[ \begin{array}{cc}
0 & -1 \\
-1 & -\gamma_k \big ( \frac{1}{\omega_k} + y_k^T H_k y_k \big ) \end{array} 
\right]
\end{equation*}
and we also have
\begin{equation*}
\begin{split}
V A^{-1} U & = \left[ \begin{array}{c}
s_k^T \\
y_k^T H_k \end{array} 
\right] H_k^{-1} \omega_k \big [ s_k \quad H_k y_k \big ] \\ & = 
\omega_k \left[ \begin{array}{c}
s_k^T \\
y_k^T H_k \end{array} 
\right] \big [ H_k^{-1} s_k \quad y_k \big ] 
 \\ & = \left[ \begin{array}{cc}
\omega_k s_k^T H_k^{-1} s_k & \omega_k s_k^T y_k \\
\omega_k y_k^T s_k & \omega_k y_k^T H_k y_k \end{array} 
\right]
\end{split}
\end{equation*}
which is just a $2 \times 2$ matrix with real entries. Now, it becomes clear that
\small
\begin{equation*}
\begin{split}
(C^{-1} + V A^{-1} U) & = \bigg ( \left[ \begin{array}{cc}
\gamma_k \big ( \frac{1}{\omega_k} + y_k^T H_k y_k \big ) & -1 \\
-1 & 0 \end{array} 
\right]^{-1} + \left[ \begin{array}{c}
s_k^T \\
y_k^T H_k \end{array} 
\right] H_k^{-1} \omega_k \big [ s_k \quad H_k y_k \big ] \bigg ) \\
 & = \left[ \begin{array}{cc}
\omega_k s_k^T H_k^{-1} s_k & -1 + \omega_k s_k^T y_k \\
-1 + \omega_k y_k^T s_k & \omega_k y_k^T H_k y_k - \gamma_k \big ( \frac{1}{\omega_k} + y_k^T H_k y_k \big ) \end{array} 
\right] .
\end{split}
\end{equation*}
\normalsize
For notational compactness, let
\begin{equation*}
D = (C^{-1} + V A^{-1} U) = \left[ \begin{array}{cc}
\omega_k s_k^T H_k^{-1} s_k & -1 + \omega_k s_k^T y_k \\
-1 + \omega_k y_k^T s_k & \omega_k y_k^T H_k y_k - \gamma_k \big ( \frac{1}{\omega_k} + y_k^T H_k y_k \big ) \end{array} 
\right]
\end{equation*} 
so
\begin{equation*}
D^{-1} = \frac{1}{\det(D)} \left[ \begin{array}{cc}
\omega_k y_k^T H_k y_k - \gamma_k \big ( \frac{1}{\omega_k} + y_k^T H_k y_k \big ) & 1 - \omega_k s_k^T y_k \\
1 - \omega_k y_k^T s_k & \omega_k s_k^T H_k^{-1} s_k \end{array} 
\right]
\end{equation*}
where the determinant of $D$ is
\small
\begin{equation*}
\begin{split}
\det(D) & = \bigg ( \omega_k y_k^T H_k y_k - \gamma_k \bigg ( \frac{1}{\omega_k} + y_k^T H_k y_k \bigg ) \bigg ) \bigg ( \omega_k s_k^T H_k^{-1} s_k \bigg ) - (1 - \omega_k y_k^T s_k)^2 \\
& = \bigg ( (\omega_k - \gamma_k ) y_k^T H_k y_k - \frac{\gamma_k}{\omega_k} \bigg ) \bigg ( \omega_k s_k^T H_k^{-1} s_k \bigg ) - (1 - \omega_k y_k^T s_k)^2
\end{split}
\end{equation*}
\normalsize
and we have used the fact that $y_k^T s_k = s_k^T y_k$, as this is a scalar quantity. Next,
\scriptsize
\begin{equation*}
\begin{split}
U \det(D) D^{-1} V & = \omega_k \big [ s_k \quad H_k y_k \big ] \left[ \begin{array}{cc}
\omega_k y_k^T H_k y_k - \gamma_k (\frac{1}{\omega_k} + y_k^T H_k y_k) & 1 - \omega_k s_k^T y_k \\
1 - \omega_k y_k^T s_k & \omega_k s_k^T H_k^{-1} s_k \end{array} 
\right] \left[ \begin{array}{c}
s_k^T \\
y_k^T H_k \end{array} 
\right] \\
& = \omega_k \big [ s_k \quad H_k y_k \big ] \left[ \begin{array}{cc}
\omega_k y_k^T H_k y_k s_k^T - \gamma_k (\frac{1}{\omega_k} + y_k^T H_k y_k) s_k^T + (1 - \omega_k s_k^T y_k) y_k^T H_k \\
(1 - \omega_k y_k^T s_k) s_k^T + \omega_k s_k^T H_k^{-1} s_k y_k^T H_k \end{array} 
\right]
\end{split} 
\end{equation*}
\normalsize
so $U \det(D) D^{-1} V$ fully expanded becomes
\begin{equation*}
\resizebox{.99\hsize}{!}{$\omega_k \bigg [ s_k \bigg ( \omega_k y_k^T H_k y_k s_k^T - \gamma_k (\frac{1}{\omega_k} + y_k^T H_k y_k) s_k^T + (1 - \omega_k s_k^T y_k) y_k^T H_k \bigg ) + H_k y_k \bigg ( (1 - \omega_k y_k^T s_k) s_k^T + \omega_k s_k^T H_k^{-1} s_k y_k^T H_k \bigg ) \bigg ] . $}
\end{equation*}
\normalsize
This looks rather ugly at the moment, but we continue by breaking the problem down further, noting that
\begin{multline*}
s_k \bigg ( \omega_k y_k^T H_k y_k s_k^T - \gamma_k \bigg ( \frac{1}{\omega_k} + y_k^T H_k y_k \bigg ) s_k^T + (1 - \omega_k s_k^T y_k) y_k^T H_k \bigg ) = \\
\bigg ( (\omega_k - \gamma_k) y_k^T H_k y_k - \frac{\gamma_k}{\omega_k} \bigg ) s_k s_k^T + (1 - \omega_k s_k^T y_k) s_k y_k^T H_k
\end{multline*}
and 
\begin{multline*}
H_k y_k \bigg ( (1 - \omega_k y_k^T s_k) s_k^T + \omega_k s_k^T H_k^{-1} s_k y_k^T H_k \bigg ) = \\ (1 - \omega_k y_k^T s_k) H_k y_k s_k^T + \omega_k H_k y_k (s_k^T H_k^{-1} s_k) y_k^T H_k .
\end{multline*}
The above intermediate results further simplify $U \det(D) D^{-1} V$ to 
\begin{equation*}
\resizebox{.95\hsize}{!}{$\omega_k \bigg [ \bigg ( (\omega_k - \gamma_k) y_k^T H_k y_k - \frac{\gamma_k}{\omega_k} \bigg ) s_k s_k^T + (1 - \omega_k s_k^T y_k) ( s_k y_k^T H_k + H_k y_k s_k^T ) + \omega_k H_k y_k (s_k^T H_k^{-1} s_k) y_k^T H_k \bigg ] . $}
\end{equation*}
Left and right multiplying the line immediately above by $A^{-1} = H_k^{-1}$ gives
\begin{equation*}
\resizebox{.95\hsize}{!}{$\omega_k \bigg [ \bigg ( (\omega_k - \gamma_k) y_k^T H_k y_k - \frac{\gamma_k}{\omega_k} \bigg ) H_k^{-1} s_k s_k^T H_k^{-1} + (1 - \omega_k s_k^T y_k) ( H_k^{-1} s_k y_k^T + y_k s_k^T H_k^{-1} ) + \omega_k y_k (s_k^T H_k^{-1} s_k) y_k^T \bigg ] $}
\end{equation*}
\normalsize
and thus, after dividing out $\det(D)$ and applying $B_{k} = H_{k}^{-1}$, we arrive at the following final formula
\begin{equation} \label{eq:SP-BFGS-Inverse-Update}
\resizebox{.88\hsize}{!}{$B_{k+1} = B_k - \frac{\omega_k \bigg [ \bigg ( (\omega_k - \gamma_k) y_k^T B_k^{-1} y_k - \frac{\gamma_k}{\omega_k} \bigg ) B_k s_k s_k^T B_k + (1 - \omega_k s_k^T y_k) ( B_k s_k y_k^T + y_k s_k^T B_k ) + \omega_k (s_k^T B_k s_k) y_k y_k^T \bigg ]}{\big ( (\omega_k - \gamma_k) y_k^T B_k^{-1} y_k - \frac{\gamma_k}{\omega_k} \big ) \big ( \omega_k s_k^T B_k s_k \big ) - (1 - \omega_k y_k^T s_k)^2} $}
\end{equation}
for the SP-BFGS inverse update, which concludes the proof.

\section{Proof of Theorem~\ref{thm:eigenvalue-bounds}}
\label{app:eigenvalue-bounds-proof}
Referring to Theorem~\ref{thm:sp-bfgs-inverse-update}, taking the trace of both sides of (\ref{eq:SP-BFGS-Inverse-Update}) and applying the linearity and cyclic invariance properties of the trace yields 
\begin{equation}
\Tr(B_{k+1}) = \kappa_1 \Tr(B_k) + \kappa_2 \norm{B_k s_k}_2^2 + 2 \kappa_3 (y_k^T B_k s_k) + \kappa_4 \norm{y_k}_2^2  
\end{equation}
where 
\begin{equation}
\kappa_1 = 1, \quad \kappa_2 = - \frac{\omega_k \hat{D}}{[ \hat{D} (\omega_k s_k^T B_k s_k) - (\hat{E})^2 ]}, \\
\end{equation}
\begin{equation}
\kappa_3 = - \frac{\omega_k \hat{E}}{[ \hat{D} (\omega_k s_k^T B_k s_k) - (\hat{E})^2 ]}, \quad \kappa_4 = - \frac{(\omega_k)^2 s_k^T B_k s_k}{[ \hat{D} (\omega_k s_k^T B_k s_k) - (\hat{E})^2 ]}
\end{equation}
with $\hat{D}$ and $\hat{E}$ defined as
\begin{equation}
\hat{D} = \bigg [ (\omega_k - \gamma_k) (y_k^T B_k^{-1} y_k) - \frac{\gamma_k}{\omega_k} \bigg ], \quad \hat{E} = (1 - \omega_k s_k^T y_k ) = \frac{2 \omega_k}{\beta_k} .
\end{equation}
We now observe that after applying some basic algebra, and recalling that $B_k$ is positive definite, one can deduce that for all $\beta_k \in [0, +\infty]$, the following inequalities hold
\begin{equation}
(\omega_k - \gamma_k) \leq 0, \quad 1 \leq \frac{\gamma_k}{\omega_k}, \quad \hat{D} \leq -1, \quad  0 \leq \frac{2 \omega_k}{\beta_k} \leq 1 .
\end{equation}
By minimizing the absolute value of the common denominator in $\kappa_2, \kappa_3$, and $\kappa_4$ using the inequalities above, we can obtain the bounds
\begin{equation}
- \frac{1}{s_k^T B_k s_k} \leq \kappa_2 \leq 0, \qquad 0 \leq \kappa_4 \leq \omega_k \leq \gamma_k 
\end{equation}
\begin{equation}
0 \leq \kappa_3 \leq \frac{2 \omega_k}{\beta_k} \frac{1}{s_k^T B_k s_k + \frac{2 \omega_k}{\beta_k} \frac{2}{\beta_k}} \leq \frac{\beta_k}{2} . 
\end{equation}
As a result, 
\begin{align}
\Tr(B_{k+1}) & \leq \Tr(B_k) + 2 \kappa_3 | y_k^T B_k s_k | + \kappa_4 \norm{y_k}_2^2 \\
 & \label{eq:alg-B_k-bound-2} \leq \Tr(B_k) + \beta_k \norm{y_k}_2 \lambda_{max}(B_k) \norm{s_k}_2 + \gamma_k \norm{y_k}_2^2 
\end{align}
and applying $\lambda_{max}(B_k) \leq \Tr(B_k)$ establishes (\ref{eq:B_k-upper-bound}). Similarly, referring to (\ref{eq:sp-bfgs-expanded}) reveals the upper bound
\begin{equation}
\Tr(H_{k+1}) \leq \Tr(H_k) + 2 \omega_k | y_k^T H_k s_k | + \bigg [ \gamma_k + \omega_k \gamma_k (y_k^T H_k y_k) \bigg ] \norm{s_k}_2^2 .
\end{equation}
To establish (\ref{eq:H_k-upper-bound}), we apply $\lambda_{max}(H_k) \leq \Tr(H_k)$ and $\omega_k \leq \gamma_k$ to the line above, and then factor. This completes the proof.

\section{Proof of Lemma~\ref{lemma:noise-dominated-region}}
\label{app:noise-dominated-region-proof}
The angle condition $\nabla \phi(x)^T H g(x) > 0$ expands to
\begin{equation}
\nabla \phi(x)^T H g(x) = \nabla \phi(x)^T H \nabla \phi(x) + \nabla \phi(x)^T H e(x) > 0 
\end{equation}
and by applying the Cauchy-Schwarz inequality and Assumption~\ref{assump:gradient-noise-bound}, we see that if 
\begin{equation}
\psi \norm{\nabla \phi(x)}_2^2 > \Psi \norm{\nabla \phi(x)}_2 \bar{\epsilon}_g 
\end{equation}
then $\nabla \phi(x)^T H g(x) > 0$. Contrapositively, if $\nabla \phi(x)^T H g(x) \leq 0$ then
\begin{equation} \label{eq:noise-region-gradient}
\norm{\nabla \phi(x)}_2 \leq \frac{\Psi \bar{\epsilon}_g}{\psi}  .
\end{equation}
As $\phi$ is m-strongly convex due to Assumption~\ref{assump:strong-convexity}, we have
\begin{equation} \label{eq:PL-inequality}
\phi^{\star} \geq \phi(x) + \min_{v} \bigg \{ \nabla \phi(x)^T v + \frac{m}{2} \norm{v}_2^2 \bigg \} = \phi(x) - \frac{1}{2 m} \norm{\nabla \phi(x)}_2^2 .
\end{equation}
Squaring (\ref{eq:noise-region-gradient}) and then combining it with (\ref{eq:PL-inequality}) gives $\mathcal{N}_1(\psi,\Psi)$, completing the proof.

\section{Proof of Theorem~\ref{thm:fixed-alpha-linear-convergence}}
\label{app:fixed-alpha-linear-convergence-proof}
As $\phi \in C^2$ by Assumption~\ref{assump:strong-convexity}, applying Taylor's theorem and using (\ref{eq:fixed-alpha-iteration}) and strong convexity gives
\begin{align*}
\phi_{k+1} & = \phi_k + \nabla \phi_k^T [ x_{k+1} - x_{k} ] + \frac{1}{2} [ x_{k+1} - x_{k} ]^T \nabla^2 \phi(u) [ x_{k+1} - x_{k} ] \\
 & \leq \phi_k - \alpha \nabla \phi_k^T H_k g_k + \frac{\alpha^2 M}{2} \norm{H_k g_k}_2^2 
\end{align*}
where $u$ is some convex combination of $x_{k+1}$ and $x_{k}$. Proceeding, note that the smallest $\mathcal{N}_{1}$ from Lemma~\ref{lemma:noise-dominated-region} occurs when $\psi = \Psi$, and in this case $\nabla \phi_k^T g_k > 0$ if $x_k \notin \mathcal{N}_{1}$. Hence, for all possible choices of $\mathcal{N}_{1}$ it is true that $\nabla \phi_k^T g_k > 0$ if $x_k \notin \mathcal{N}_{1}$. Combining this with (\ref{eq:inner-product-proportional}) gives 
\begin{equation} \label{eq:inner-product-proportional-lower-bounded}
\nabla \phi_k^T H_k g_k \geq \psi \nabla \phi_k^T g_k > 0
\end{equation}
if $x_k \notin \mathcal{N}_{1}$. With (\ref{eq:inner-product-proportional-lower-bounded}) in hand, continuing to bound terms gives
\small
\begin{align*}
\phi_{k+1} & \leq \phi_k - \alpha \psi \nabla \phi_k^T [ \nabla \phi_k + e_k ] + \frac{\alpha^2 \Psi^2 M}{2} \norm{\nabla \phi_k + e_k}_2^2 \\
 & = \phi_k - \alpha \Psi \bigg ( \frac{\psi}{\Psi} - \frac{\alpha \Psi M}{2} \bigg ) \norm{\nabla \phi_k}_2^2 - \alpha \Psi \bigg ( \frac{\psi}{\Psi} - \alpha \Psi M \bigg ) \nabla \phi_k^T e_k + \frac{\alpha^2 \Psi^2 M}{2} \norm{e_k}_2^2 \\
 & \leq \phi_k - \alpha \Psi \bigg ( \frac{\psi}{\Psi} - \frac{\alpha \Psi M}{2} \bigg ) \norm{\nabla \phi_k}_2^2 + \alpha \Psi \bigg ( \frac{\psi}{\Psi} - \alpha \Psi M \bigg ) \norm{\nabla \phi_k}_2 \norm{e_k}_2 + \frac{\alpha^2 \Psi^2 M}{2} \norm{e_k}_2^2 \\
& \leq \phi_k - \alpha \Psi \bigg ( \frac{\psi}{\Psi} - \frac{\alpha \Psi M}{2} \bigg ) \norm{\nabla \phi_k}_2^2 + \alpha \Psi \bigg ( \frac{\psi}{\Psi} - \alpha \Psi M \bigg ) \bigg [ \frac{1}{2} \norm{\nabla \phi_k}_2^2 + \frac{1}{2} \norm{e_k}_2^2 \bigg ] \\
& \qquad + \frac{\alpha^2 \Psi^2 M}{2} \norm{e_k}_2^2 
\end{align*}
\normalsize
where the last inequality follows from expanding 
\small
\begin{equation}
0 \leq \bigg ( \frac{1}{\sqrt{2}} \norm{\nabla \phi_k}_2 - \frac{1}{\sqrt{2}} \norm{e_k}_2 \bigg )^2 = \frac{1}{2} \norm{\nabla \phi_k}_2^2 - \norm{\nabla \phi_k}_2 \norm{e_k}_2 + \frac{1}{2} \norm{e_k}_2^2
\end{equation}
\normalsize
and using (\ref{eq:alpha-fixed-size-lower-upper-bound}). Simplifying the last inequality reveals that 
\begin{equation} \label{eq:simplified-worst-case-descent-bound}
\phi_{k+1} \leq \phi_k - \frac{\alpha \psi}{2} \norm{\nabla \phi_k}_2^2 + \frac{\alpha \psi}{2} \norm{e_k}_2^2 .
\end{equation} 
Since $\phi$ is m-strongly convex by Assumption~\ref{assump:strong-convexity}, we can apply 
\begin{equation} \label{eq:PL-inequality-shortscript}
\norm{\nabla \phi_k}_2^2 \geq 2 m ( \phi_k - \phi^{\star} )
\end{equation}
as shown in the proof of Lemma~\ref{lemma:noise-dominated-region} (see Appendix~\ref{app:noise-dominated-region-proof}), which combined with (\ref{eq:simplified-worst-case-descent-bound}) and Assumption~\ref{assump:gradient-noise-bound} gives
\begin{equation}
\phi_{k+1} \leq \phi_k - \alpha \psi m ( \phi_k - \phi^{\star} ) + \frac{\alpha \psi}{2} \bigg ( \frac{\Psi \bar{\epsilon}_g}{\psi} \bigg )^2 . 
\end{equation}
Subtracting $\phi^{\star}$ from both sides, we get
\begin{equation}
\phi_{k+1} - \phi^{\star} \leq (1 - \alpha \psi m) (\phi_k - \phi^{\star} ) + \frac{\alpha \psi}{2} \bigg ( \frac{\Psi \bar{\epsilon}_g}{\psi} \bigg )^2 
\end{equation}
which, by subtracting $\frac{1}{2 m} (\frac{\Psi \bar{\epsilon}_g}{\psi})^2$ from both sides and simplifying, gives 
\small
\begin{align*}
\phi_{k+1} - \phi^{\star} - \frac{1}{2 m} \bigg ( \frac{\Psi \bar{\epsilon}_g}{\psi} \bigg )^2 & \leq (1 - \alpha \psi m) (\phi_k - \phi^{\star} ) + \frac{\alpha \psi}{2} \bigg ( \frac{\Psi \bar{\epsilon}_g}{\psi} \bigg )^2 - \frac{1}{2 m} \bigg ( \frac{\Psi \bar{\epsilon}_g}{\psi} \bigg )^2 \\
& = (1 - \alpha \psi m) (\phi_k - \phi^{\star} ) + (\alpha \psi m - 1) \frac{1}{2 m} \bigg ( \frac{\Psi \bar{\epsilon}_g}{\psi} \bigg )^2 \\
& = (1 - \alpha \psi m) \bigg ( \phi_k - \bigg [ \phi^{\star} + \frac{1}{2 m} \bigg ( \frac{\Psi \bar{\epsilon}_g}{\psi} \bigg )^2 \bigg ] \bigg ) 
\end{align*}
\normalsize
thus establishing the Q-linear result (\ref{eq:fixed-alpha-Q-linear-rate}). We obtain (\ref{eq:fixed-alpha-R-linear-rate}) by recursively applying the worst case bound in (\ref{eq:fixed-alpha-Q-linear-rate}), noting that in the worst case if $x_0 \notin \mathcal{N}_{1}$, then the sequence of iterates $\{ x_k \}$ remains outside of $\mathcal{N}_{1}$, only approaching $\mathcal{N}_{1}$ in the limit $k \rightarrow \infty$.

\section{Extended Numerical Experiments}
\label{app:extended-numerical-experiments}
Tables~\ref{tab:SPBFGS-alg-comp-grad-noise-only} and~\ref{tab:BFGS-alg-comp-grad-noise-only} compare the performance of SP-BFGS vs. BFGS on the $32$ CUTEst test problems with only gradient noise present (i.e. $\bar{\epsilon}_f = 0$). Gradient noise was generated using $\bar{\epsilon}_g = 10^{-4} \norm{\nabla \phi(x^0)}_2$, where the starting point $x^0$ varies by CUTEst problem, to ensure that noise does not initially dominate gradient evaluations. Overall, SP-BFGS outperforms BFGS on approximately $80 \%$ of the CUTEst problems with only gradient noise present, and performs at least as good as BFGS on approximately $95 \%$ of these problems.

\newpage
\begin{table}[h]
\begin{center}
  {\def\arraystretch{1.5}\tabcolsep=3pt
  \begin{tabular}{ | c | c | c | c | c | c | c | }
    \hline
        \multicolumn{7}{ | c | }{ \textbf{SP-BFGS With Gradient Noise Only}} \\
    \hline
    \textit{Problem} & \textit{Dim.} & \textit{Mean($\Delta_{opt}$)} & \textit{Median($\Delta_{opt}$)} & \textit{Min($\Delta_{opt}$)} & \textit{Max($\Delta_{opt}$)} & \textit{$s^2(\Delta_{opt})$} \\ 
    \hline
    ARGTRGLS & 200 & -9.6E-02 & -9.6E-02 & -1.0E-01 & -8.5E-02 & 1.9E-05 \\
    \hline
    ARWHEAD & 500 & -2.8E+00 & -2.8E+00 & -2.8E+00 & -2.7E+00 & 1.7E-03 \\
    \hline
    BEALE & 2 & -1.4E+01 & -1.4E+01 & -1.6E+01 & -7.0E+00 & 4.1E+00 \\
    \hline
    BOX3 & 3 & -6.7E+00 & -6.5E+00 & -1.1E+01 & -6.3E+00 & 6.3E-01 \\
    \hline
    BOXPOWER & 100 & -2.7E+00 & -2.7E+00 & -3.1E+00 & -2.3E+00 & 4.6E-02 \\
    \hline
	BROWNBS & 2 & -4.5E+00 & -5.9E+00 & -8.0E+00 & 1.1E+00 & 8.4E+00 \\
	\hline
	BROYDNBDLS & 50 & -5.4E+00 & -5.4E+00 & -5.9E+00 & -5.0E+00 & 3.4E-02 \\
	\hline
	CHAINWOO & 100 & 1.6E+00 & 1.7E+00 & 7.6E-02 & 2.1E+00 & 1.5E-01 \\
	\hline
	CHNROSNB & 50 & -3.2E+00 & -3.0E+00 & -4.9E+00 & -2.6E+00 & 4.5E-01 \\
	\hline
	COATING & 134 & 3.4E-01 & 3.4E-01 & 1.8E-01 & 4.2E-01 & 3.1E-03 \\
	\hline
	COOLHANSLS & 9 & -9.4E-01 & -9.4E-01 & -1.2E+00 & -4.8E-01 & 4.2E-02 \\
	\hline
	CUBE & 2 & -2.7E+00 & -2.5E+00 & -5.8E+00 & -1.7E+00 & 7.5E-01 \\
	\hline
	CYCLOOCFLS & 20 & -7.4E+00 & -7.2E+00 & -9.3E+00 & -5.9E+00 & 8.1E-01 \\
	\hline
	EXTROSNB & 10 & -5.1E+00 & -5.2E+00 & -5.3E+00 & -4.7E+00 & 3.0E-02 \\
	\hline
	FMINSRF2 & 64 & -8.6E+00 & -8.7E+00 & -8.8E+00 & -8.1E+00 & 3.4E-02 \\
	\hline
	GENHUMPS & 5 & -2.7E+00 & -2.6E+00 & -5.2E+00 & -1.0E+00 & 1.1E+00 \\
	\hline
	GENROSE & 5 & -1.2E+01 & -1.2E+01 & -1.4E+01 & -8.9E+00 & 2.0E+00 \\
	\hline
	HEART6LS & 6 & 1.0E+00 & 1.2E+00 & -1.8E+00 & 1.2E+00 & 5.0E-01 \\
	\hline
	HELIX & 3 & -5.7E+00 & -5.9E+00 & -8.7E+00 & -3.4E+00 & 1.4E+00 \\
	\hline
	MANCINO & 30 & -1.0E+00 & -1.0E+00 & -1.4E+00 & -7.0E-01 & 3.7E-02 \\
	\hline
	METHANB8LS & 31 & -3.6E+00 & -3.6E+00 & -4.0E+00 & -3.3E+00 & 3.1E-02 \\
	\hline
	MODBEALE & 200 & 1.2E+00 & 1.2E+00 & 3.8E-01 & 1.9E+00 & 1.8E-01 \\
	\hline
	NONDIA & 10 & -3.5E-03 & -3.6E-03 & -4.3E-03 & -1.1E-03 & 6.6E-07 \\
	\hline
	POWELLSG & 4 & -5.7E+00 & -5.3E+00 & -9.3E+00 & -4.0E+00 & 1.6E+00 \\
	\hline
	POWER & 10 & -3.5E+00 & -3.5E+00 & -4.4E+00 & -2.8E+00 & 1.3E-01 \\
	\hline
	ROSENBR & 2 & -1.1E+01 & -1.2E+01 & -1.4E+01 & -5.1E+00 & 4.4E+00 \\
	\hline
	ROSENBRTU & 2 & -1.9E+01 & -1.9E+01 & -2.2E+01 & -1.7E+01 & 1.1E+00 \\
	\hline
	SBRYBND & 500 & 3.9E+00 & 3.9E+00 & 3.9E+00 & 3.9E+00 & 2.0E-05 \\
	\hline
	SINEVAL & 2 & -1.3E+01 & -1.3E+01 & -1.8E+01 & -1.1E+01 & 3.3E+00 \\
	\hline
	SNAIL & 2 & -1.5E+01 & -1.6E+01 & -1.8E+01 & -1.2E+01 & 1.6E+00 \\
	\hline
	SROSENBR & 1000 & -9.7E-01 & -9.7E-01 & -1.3E+00 & -4.8E-01 & 3.2E-02 \\
	\hline
	VIBRBEAM & 8 & 1.6E+00 & 1.6E+00 & 1.2E+00 & 2.8E+00 & 9.1E-02 \\
	\hline
	\end{tabular}
  }
\end{center}
\caption{Performance of SP-BFGS on $32$ selected CUTEst test problems with noise added to gradient evaluations only (i.e. $\bar{\epsilon}_f = 0$). The number of objective function evaluations is fixed at $2000$. $\Delta_{opt} \coloneqq \log_{10}(\phi_{best} - \phi^{\star})$ measures the optimality gap, where $\phi_{best}$ denotes the smallest value of the true function $\phi$ measured at any point during an algorithm run. Statistics are calculated from a sample of $30$ runs per algorithm, and the \textit{Dim.} column gives the problem dimension. The SPBFGS penalty parameter was set as $\beta_k = \frac{10^8}{\bar{\epsilon}_g} \norm{s_k}_2 + 10^{-10}$. For each problem, gradient noise was generated using $\bar{\epsilon}_g = 10^{-4} \norm{\nabla \phi(x^0)}_2$, where the starting point $x^0$ varies by CUTEst problem.  }
\label{tab:SPBFGS-alg-comp-grad-noise-only}
\end{table}

\newpage
\begin{table}[h]
\begin{center}
  {\def\arraystretch{1.5}\tabcolsep=3pt
  \begin{tabular}{ | c | c | c | c | c | c | c | }
    \hline
        \multicolumn{7}{ | c | }{ \textbf{BFGS With Gradient Noise Only}} \\
    \hline
    \textit{Problem} & \textit{Dim.} & \textit{Mean($\Delta_{opt}$)} & \textit{Median($\Delta_{opt}$)} & \textit{Min($\Delta_{opt}$)} & \textit{Max($\Delta_{opt}$)} & \textit{$s^2(\Delta_{opt})$} \\ 
    \hline
    ARGTRGLS & 200 & -9.2E-02 & -9.3E-02 & -9.9E-02 & -8.2E-02 & 1.7E-05 \\
    \hline
    ARWHEAD & 500 & -2.5E+00 & -2.5E+00 & -2.6E+00 & -2.5E+00 & 7.6E-04 \\
    \hline
    BEALE & 2 & -8.3E+00 & -8.5E+00 & -1.2E+01 & -5.8E+00 & 3.9E+00 \\
    \hline
    BOX3 & 3 & -6.4E+00 & -6.4E+00 & -6.6E+00 & -6.3E+00 & 2.3E-03 \\
    \hline
    BOXPOWER & 100 & -2.8E+00 & -2.8E+00 & -3.3E+00 & -2.4E+00 & 6.1E-02 \\
	\hline
	BROWNBS & 2 & 6.8E-02 & 1.0E+00 & -8.2E+00 & 3.6E+00 & 1.0E+01 \\
	\hline
	BROYDNBDLS & 50 & -5.1E+00 & -5.1E+00 & -5.3E+00 & -4.9E+00 & 1.5E-02 \\
	\hline
	CHAINWOO & 100 & 1.7E+00 & 1.8E+00 & 1.1E+00 & 2.2E+00 & 5.8E-02 \\
	\hline
	CHNROSNB & 50 & -2.9E+00 & -2.7E+00 & -4.5E+00 & -2.1E+00 & 3.8E-01 \\
	\hline
	COATING & 134 & 3.6E-01 & 3.7E-01 & 2.1E-01 & 4.2E-01 & 2.6E-03 \\
	\hline
	COOLHANSLS & 9 & -5.6E-01 & -6.4E-01 & -1.3E+00 & 1.9E-01 & 1.9E-01 \\
	\hline
	CUBE & 2 & -1.1E+00 & -1.1E+00 & -1.8E+00 & -9.6E-01 & 5.6E-02 \\
	\hline
	CYCLOOCFLS & 20 & -6.5E+00 & -6.5E+00 & -8.3E+00 & -5.1E+00 & 5.4E-01 \\
	\hline
	EXTROSNB & 10 & -5.1E+00 & -5.1E+00 & -5.3E+00 & -4.9E+00 & 8.1E-03 \\
	\hline
	FMINSRF2 & 64 & -8.2E+00 & -8.2E+00 & -8.7E+00 & -7.3E+00 & 1.4E-01 \\
	\hline
	GENHUMPS & 5 & -1.5E+00 & -1.2E+00 & -4.0E+00 & -2.8E-01 & 8.4E-01 \\
	\hline
	GENROSE & 5 & -6.8E+00 & -6.7E+00 & -8.7E+00 & -5.9E+00 & 4.8E-01 \\
	\hline
	HEART6LS & 6 & 1.2E+00 & 1.2E+00 & 1.2E+00 & 1.2E+00 & 1.9E-04 \\
	\hline
	HELIX & 3 & -4.8E+00 & -4.6E+00 & -8.1E+00 & -2.6E+00 & 2.1E+00 \\
	\hline
	MANCINO & 30 & -8.3E-01 & -8.8E-01 & -1.2E+00 & -3.3E-01 & 5.3E-02 \\
	\hline
	METHANB8LS & 31 & -3.5E+00 & -3.4E+00 & -3.9E+00 & -3.3E+00 & 2.8E-02 \\
	\hline
	MODBEALE & 200 & 1.0E+00 & 1.1E+00 & -6.2E-01 & 2.1E+00 & 3.6E-01 \\
	\hline
	NONDIA & 10 & 1.2E-03 & 1.3E-03 & -4.4E-03 & 1.3E-02 & 2.2E-05 \\
	\hline
	POWELLSG & 4 & -5.3E+00 & -5.0E+00 & -8.0E+00 & -3.6E+00 & 1.6E+00 \\
	\hline
	POWER & 10 & -3.4E+00 & -3.4E+00 & -4.3E+00 & -2.8E+00 & 1.4E-01 \\
	\hline
	ROSENBR & 2 & -6.1E+00 & -5.9E+00 & -1.0E+01 & -3.7E+00 & 2.9E+00 \\
	\hline
	ROSENBRTU & 2 & -1.5E+01 & -1.5E+01 & -1.8E+01 & -1.4E+01 & 1.6E+00 \\
	\hline
	SBRYBND & 500 & 3.9E+00 & 3.9E+00 & 3.9E+00 & 3.9E+00 & 2.0E-05 \\
	\hline
	SINEVAL & 2 & -1.2E+01 & -1.3E+01 & -1.7E+01 & -8.5E+00 & 4.0E+00 \\
	\hline
	SNAIL & 2 & -1.1E+01 & -1.1E+01 & -1.6E+01 & -8.2E+00 & 3.5E+00 \\
	\hline
	SROSENBR & 1000 & -9.1E-01 & -8.8E-01 & -1.3E+00 & -5.1E-01 & 3.1E-02 \\
	\hline
	VIBRBEAM & 8 & 1.7E+00 & 1.6E+00 & 1.4E+00 & 2.6E+00 & 1.0E-01 \\
	\hline
	\end{tabular}
  }
\end{center}
\caption{Performance of BFGS on $32$ selected CUTEst test problems with noise added to gradient evaluations only (i.e. $\bar{\epsilon}_f = 0$). The number of objective function evaluations is fixed at $2000$. $\Delta_{opt} \coloneqq \log_{10}(\phi_{best} - \phi^{\star})$ measures the optimality gap, where $\phi_{best}$ denotes the smallest value of the true function $\phi$ measured at any point during an algorithm run. Statistics are calculated from a sample of $30$ runs per algorithm, and the \textit{Dim.} column gives the problem dimension. For each problem, gradient noise was generated using $\bar{\epsilon}_g = 10^{-4} \norm{\nabla \phi(x^0)}_2$, where the starting point $x^0$ varies by CUTEst problem. }
\label{tab:BFGS-alg-comp-grad-noise-only}
\end{table}

\end{document}